\documentclass[reqno]{amsart}

\usepackage{amsfonts}
\usepackage{amsthm}
\usepackage{amstext}
\usepackage{amsmath}
\usepackage{amscd}
\usepackage{amssymb}
\usepackage[mathscr]{eucal}
\usepackage{mathrsfs}
\usepackage{epsf}
\usepackage{array}
\usepackage{url}
\usepackage{enumerate}
\usepackage{enumitem}
\usepackage{graphicx}
\usepackage{accents}
\usepackage{hyperref}
\usepackage{multirow}

\newcommand{\NNN}{\mathbb{N}}
\newcommand{\RRR}{\mathbb{R}}
\newcommand{\card}{\#}    
\newcommand{\AAa}{A}      
      
\newcommand{\CCc}{\mathcal{C}}      
\newcommand{\eps}{\varepsilon}
\newcommand{\ubar}[1]{\text{\b{$#1$}}}

\newcommand{\iii}{\mathcal{I}}

\newcommand{\kkk}{\mathcal{K}}

\newcommand{\dens}{{d}}

\newcommand{\recog}{\varrho}

\newcommand{\lang}{\mathscr{L}}    
\newcommand{\recplot}{\mathcal{R}}

\newcommand{\emptyword}{o}
\newcommand{\word}[3]{#1_{[ #2 , #3)}}
\newcommand{\linepattern}[3]{#2^{#1 #3}}
\newcommand{\indexoflines}[3]{\mathcal{I}_{#1}^{#2 #3}}
\newcommand{\abs}[1]{|#1|}

 \makeatletter
 \def\@seccntformat#1{\csname the#1\endcsname.\quad}
 \makeatother

\theoremstyle{plain}
\newtheorem{theorem}{Theorem}[section]
\newtheorem{proposition}[theorem]{Proposition}
\newtheorem{corollary}[theorem]{Corollary}
\newtheorem{lemma}[theorem]{Lemma}

\theoremstyle{definition}
\newtheorem{definition}[theorem]{Definition}

\theoremstyle{remark}
\newtheorem{remark}[theorem]{Remark}
\newtheorem*{notation}{Notation}

\theoremstyle{remark}
\newtheorem{example}[theorem]{Example}

\numberwithin{equation}{section}

\begin{document}

\bibliographystyle{amsplain}

\title[Symbolic recurrence plot for uniform binary substitutions]
 {Symbolic recurrence plot for uniform binary substitutions}

\author[M. Pol\'akov\'a]{Miroslava Pol\'akov\'a}
\address{
  Department of~Mathematics, Faculty of~Natural Sciences, Matej Bel University, Tajovsk\'eho~40,
  Bansk\'a Bystrica, Slovakia
}
\email{miroslava.sartorisova@umb.sk}



\begin{abstract}
Diagonal lines in symbolic recurrence plots are closely related to the identification and characterization of specific biprolongable words within a sequence.
In this paper we focus on the recurrence plot of a fixed point of a uniform binary substitution.
We show that, if the substitution is primitive and aperiodic, 
the set of all diagonal line lengths of the recurrence plot has zero density. 
However, if a line of a specific length exists in the recurrence plot, 
the density of (the set of starting points of) all diagonal lines with that length is strictly positive.
On the other hand, we demonstrate that the recurrence plot of a non-primitive substitution contains lines of any given length. Nonetheless, for any given length, the density of lines with that length is zero.
\end{abstract}
\maketitle
\thispagestyle{empty}

\section{Introduction}\label{S:intro}
Recurrence plots, introduced by Eckmann, Kamphorst, and Ruelle \cite{eckmann1987recurrence} in 1987, provide a visual representation of recurrences in a dynamical system, contributing to a deeper understanding of the system's behavior.
Diagonal lines help to visualize the structure and repetitions within the data, making recurrence plots a powerful tool to analyze time series and dynamical systems.
Length and density of diagonal lines in the recurrence plot provide valuable insights into the dynamics of the system
and play a fundamental role in recurrence quantification analysis \cite{zbilut1992embeddings,webber1994dynamical},
see also \cite{webber2015recurrence,marwan2023trends}. Since we deal with binary sequences, we employ the so-called 
symbolic recurrence plots \cite{faure2010recurrence}, which are sufficient for a complete recurrence quantification analysis
of symbolic sequences.

Substitution sequences form a specific class of sequences; they are generated by substitution rules applied to an initial letter or word.
Sequences generated as fixed points of substitutions offer a wide range of examples,
from trivial to complex ones.
In this paper, our focus is on constant length (uniform) substitutions over a binary alphabet.
To study these sequences, we utilize diagonal lines in recurrence plots,
which closely relate to identifying words in the sequence that 
are left-and-right biprolongable, that is, 
words which can be extended by two coordinatewise different pairs of letters as a prefix and a suffix at the same time.
Words that generate these diagonal lines in recurrence plots, along with their associated prefix and suffix, will be referred to as inner line-patterns, which we now briefly introduce.

Let $\zeta$ be a primitive aperiodic substitution of constant length on the binary alphabet $\AAa=\{0,1\}$.
If $\zeta(0)$ starts with $0$, then $\zeta$ admits a fixed point $x = \zeta^\infty(0)\in\AAa^{\NNN_0}$.
An (infinite) \emph{symbolic recurrence plot} $\recplot=\recplot(x,\infty,\eps_0=1/2)$
of $x$ is the infinite $01$-matrix $(\recplot_{ij})_{i,j\in\NNN_0}$ 
such that $\recplot_{ij}=1$ if and only if $x_i=x_j$.
A diagonal line in a symbolic recurrence plot corresponds to a repetition
of a word in $x$ that cannot be prolonged from either the left or the right.
If a line does not start at the boundary of the recurrence plot, it is called an inner line.
For the details, see Subsection~\ref{SUBS:RecPlot}.

Let $\alpha$ ($\beta$) denote the length of the longest common prefix (suffix, respectively) of $\zeta(0)$ and $\zeta(1)$.
An \emph{(inner) line-pattern} $P=\linepattern{a}{w}{b}$ is a triple of a nonempty word $w$ 
and letters $a,b \in \AAa$ such that there exist $i,j \in \NNN$ satisfying 
\begin{equation*}
	awb = \word{x}{i-1}{i+\abs{w}+1}
	\quad\text{and}\quad
	\bar{a}w\bar{b} = \word{x}{j-1}{j+\abs{w}+1}
\end{equation*}
(refer to Definition~\ref{DEF:line-pattern} and accompanying notes); here $\bar{0}=1$ and $\bar{1}=0$. The set of all such pairs $(i,j)$ that satisfy these conditions is denoted by $\iii(P)$. Additionally, the length of $w$ is denoted by $\abs{P}$. 

The key concept is the notion of induced line-patterns. A line-pattern $P=\linepattern{a}{w}{b}$ is said to be \emph{induced} by a line-pattern $Q=\linepattern{c}{v}{d}$ if
\begin{equation*}
	awb = \word{\zeta(c)}{q-\beta-1}{q}\zeta(v)\word{\zeta(d)}{0}{\alpha+1}
	\quad\text{and}\quad
	\bar{a}w\bar{b} = \word{\zeta(\bar c)}{q-\beta-1}{q}\zeta(v)\word{\zeta(\bar d)}{0}{\alpha+1}
\end{equation*}	
(see Definition~\ref{DEF:induced} and Proposition~\ref{P:substitution-of-line-patterns}). 
As every inner line-pattern induces an infinite chain of line-patterns with increasing lengths,
every inner line in the recurrence plot of $x = \zeta^\infty(0)$ induces infinitely many longer and longer lines.
To distinguish whether a given line (line-pattern) was induced by some shorter one, we need to employ recognizability.
Recognizability enables us to uniquely decompose any (sufficiently long) word in $x$
in an appropriate way
and thus, for every sufficiently long line (line-pattern), find out its ``predecessor'', that is,
a line (line-pattern) from which it was induced (see Subsection~\ref{SUBS:recognizability}).
The following proposition summarizes these results on ``desubstitution'' of line-patterns. There, 
$\recog \in \NNN_0$ is the smallest integer larger than $\alpha + \beta$ and such that every word in $x$ of length (at least) $\recog$ is recognizable. A line-pattern $P=\linepattern{a}{w}{b}$ is $\recog$-recognizable if
the length of $w$ is at least $\recog$ (and hence $w$ is recognizable).

\begin{proposition}[Desubstitution of inner line-patterns]\label{P:desubstitution-of-line-patterns}
	Let $P$ be a $\recog$-recognizable inner line-pattern.
	Then there exists a \emph{unique} inner line-pattern $Q$ such that $P$ is induced by $Q$.
	Moreover,
	\begin{equation*}	
	\iii(P) = q\iii(Q) - \beta
	\quad\text{and}\quad
	\abs{P} = q\abs{Q} + \alpha + \beta.
	\end{equation*}
\end{proposition}

Thus, every sufficiently long line-pattern $P$ is induced by exactly one shorter line-pattern $Q$ in such a way 
that every line in the recurrence plot given by $P$ is induced by unique line given by $Q$.
Thus, to obtain a complete description of the symbolic recurrence plot of a substitution sequence, 
it suffices to describe lines of length smaller than $\recog$.

Let $\ell\in\NNN$ and let $\kkk_\ell$ be the set of all pairs $(i,j)$ of positive integers such that 
an inner line of length $\ell$ starts at it.
We are interested in the densities of the sets $\kkk_\ell$, as they play a key role in asymptotic
analysis of several basic recurrence quantifiers, see e.g.~\cite{vspitalsky2018recurrence}.
To obtain the density $\dens(\kkk_\ell)$, one needs to use the following facts (refer to Subsection~\ref{subs:density-of-inner-lines} for the details).
First, for any given $\ell$ one can find finitely many line-patterns $P^j$ such that the set $\kkk_\ell$ 
is a disjoint union of the sets $\iii(P^j)$. Second, if $P=\linepattern{a}{w}{b}$ is a line-pattern, the density of $\dens(\iii(P))$ exists, is strictly positive, and is equal to the product of measures of cylinders
$[awb]$ and $[\bar{a}w\bar{b}]$. These two facts enables us to calculate the density $\dens(\kkk_\ell)$ for any given length $\ell$. Finally, if $\ell$ is sufficiently large (that is, $\ell\ge\recog$), 
we can derive the density of $\kkk_\ell$ from that of $\kkk_{\ell'}$ with smaller $\ell'$
using Proposition~\ref{P:desubstitution-of-line-patterns} instead.
This leads to the following theorem about the density of (the set of starting points of) inner lines in the infinite symbolic recurrence plot.

\begin{theorem}\label{THM:density}
	Let $\zeta:\AAa\to\AAa^*$ be a primitive aperiodic binary substitution of constant length $q\ge 2$ such that
	$\zeta(0)$ starts with $0$,
	and $x=\zeta^\infty(0)$ be a unique fixed point of $\zeta$ starting with $0$. 
	Let $\kkk_\ell$ ($\ell\in\NNN$) be the set of starting points of inner-lines of length $\ell$ 
	in infinite symbolic recurrence plot $\recplot(x,\infty,\eps_0)$. 
	Then, for arbitrary $\ell\in\NNN$, the following is true:
	\begin{enumerate}
		\item \label{Case1-in-THM:density} 
		$\kkk_\ell = \emptyset$ if and only if $\dens(\kkk_\ell) = 0$;

		\item \label{Case2-in-THM:density} 
		there are $r\in\NNN_0$ and unique (up to a permutation) inner line-patterns
		$P^1,\dots,P^r$ such that
		\begin{equation}
			\kkk_\ell = \bigsqcup_{j=1}^r \iii(P^j)
			\quad\text{and}\quad
			\dens(\kkk_\ell) = \sum_{j=1}^r \dens(\iii(P^j));
		\end{equation}
		further, if $P^j=\linepattern{a}{w}{b}$ then $\dens(\iii(P^j))=\mu([awb])\cdot \mu([\bar aw\bar b])>0$,
		where $\mu$ is the unique invariant measure of the $\zeta$-shift;
 		
		\item \label{Case3-in-THM:density} 
		if $\ell \geq \recog$ and 
		$\kkk_\ell \neq \emptyset$, there are unique positive integers $\ell_0$ and $k$ such that
		$\ell_0 < \recog\le q\ell_0+\alpha+\beta$, $\ell = q^k \ell_0 + c(q^k-1) $ and
		$$ 
		\dens (\kkk_\ell)
		= q^{-2k}\dens(\kkk_{\ell_0}) 
		= \frac{(\ell_0 + c)^2 }{(\ell + c)^2} \dens(\kkk_{\ell_0}) ,
		$$
		where $c = (\alpha+\beta)/(q-1) \in [ 0, 1]$.	
	\end{enumerate}
	Consequently, the set of all integers $\ell$ with $\kkk_\ell\ne\emptyset$ is a zero density subset of $\NNN_0$.
\end{theorem}

To summarize, for a primitive aperiodic substitution in which $\zeta(0)$ starts with $0$, we only have a 
``small amount'' of possible lengths $\ell$ of lines in the recurrence plot; however, once $\kkk_\ell$ is nonempty, 
it always has a positive density. (The density can be explicitly calculated, as is shown in Section~\ref{S:densities}.)
In contrast to primitive substitutions, in Section~\ref{S:non-primitive} we demonstrate that, for non-primitive substitutions, the set $\kkk_\ell$ is always nonempty (even infinite), but has zero density for every $\ell$.

\begin{theorem}\label{THM:nonprimit}
Let $\zeta$ be a non-primitive binary substitution of constant length $q\ge 2$ 
such that $\zeta(0)$ starts with $0$
and such that the fixed point $x = \zeta^\infty(0)$ of
$\zeta$ is not eventually $\sigma$-periodic.
Let $\kkk_\ell$ ($\ell\in\NNN$) be the set of starting points of inner-lines of length $\ell$ 
in infinite symbolic recurrence plot $\recplot(x,\infty,\eps_0)$. 
Then, for every $\ell\in\NNN$, the set $\kkk_\ell$ is infinite (hence nonempty) and has zero density.
\end{theorem}

\medskip

The paper is organized as follows.
In the following section, we recall some standard facts about symbolic sequences, substitutions, and (symbolic) recurrence plots. We also discuss the recognizability of substitutions.
Then, in Sections~\ref{S:lengths-and-density}--\ref{S:densities}, we focus on primitive aperiodic substitutions. 
In Section~\ref{S:lengths-and-density}, we introduce the concept of (induced) line-patterns. Additionally, we explore the density of the set of starting points of inner lines in infinite recurrence plots. We deal also with the so-called 
 $0$-boundary line-patterns, see Subsection~\ref{SUBS:0-boundary}.
In Section~\ref{S:lengths-and-density-finite}, we compare finite and infinite symbolic recurrence plots
and show that the number of different lengths of lines in finite recurrence plots as well as the number of recurrences in the so-called $n$-boundary lines is fairly small.
An algorithm for computing the densities $\dens(\kkk_\ell)$ for these substitutions is presented in Section~\ref{S:densities}, with the details postponed to Appendix~\ref{APP:algorithm}. 
In Sections~\ref{S:non-primitive} and \ref{S:other-subs}, we discuss substitutions that do not satisfy primitivity or aperiodicity.
Finally, in Appendix~\ref{APP:1cutting} we describe the connection between recognizability of a word and uniqueness of a $1$-cutting of it.

\section{Preliminaries}\label{S:prelim}

The set of non-negative (positive) integers is denoted by $\NNN_0$ ($\NNN$).
The cardinality of a set $B$ is denoted by $\card B$.
When no confusion can arise, the set of consecutive
integers $\{ a, a+1, \dots, b-1 \}$ for any $a < b$ from $\NNN_0$ is denoted by $[a, b)$;
if $a \geq b$ then $[a, b) = \emptyset$. 
For $B \subseteq \RRR^{k}$ and $c, d \in \RRR$ put $cB+d = \{ cb+d\colon b \in B \}$.

Let $k \in \NNN$ and $M \subseteq \NNN_0^k$.
Then the \emph{upper} and \emph{lower (asymptotic) densities} of $M$ are given by
$$ 
	\bar \dens (M) = \limsup\limits_{n \to \infty} \frac{1}{n^k} \card \big(M \cap [0, n)^k\big) 
	\quad\text{and}\quad 
	\ubar{\dens}(M) = \liminf\limits_{n \to \infty} \frac{1}{n^k} \card \big(M \cap [0, n)^k \big).
$$
If $\ubar{\dens}(M) = \bar \dens (M)$ we say that the (asymptotic) density $\dens(M)$ of $M$
exists and is equal to this common value.
Clearly, if densities of disjoint sets $M,N \subseteq \NNN_0^k$ exist then
\begin{equation}\label{EQ:density-props}
	\dens(M\sqcup N)=\dens(M)+\dens(N)
	\quad\text{and}\quad
	\dens(aM+b)=a^{-k} \dens(M)
\end{equation}
for every integers $a>0$ and $b$.
Further, for any $M \subseteq \mathbb{N}_0^k$ and $N \subseteq \mathbb{N}_0^l$, $\dens(M \times N) = \dens(M)\dens(N)$ provided $\dens(M)$ and $\dens(N)$ exist.

\subsection{The metric space of symbolic sequences}
An \emph{alphabet} $\AAa$ is a nonempty finite set; elements of $\AAa$ are called \emph{letters}.
Put 
\begin{equation*}
	\Sigma=\AAa^{\NNN_0}.
\end{equation*}
Members $x = x_0 x_1 x_2 \dots $ of $\Sigma$ are called \emph{sequences}.
The \emph{subword of x of length $\ell$ starting at index $i$}\footnote{We use the terms subword and index instead of factor and rank, respectively.} is an $\ell$-word $x_i x_{i+1} \dots x_{i+\ell-1}$ and will be denoted by $\word{x}{i}{i+\ell}$.

Metric $\rho$ on $\Sigma$ is defined for every $x, y \in \Sigma$ by $\rho(x, y) = 0$
if $x=y$ and 
\begin{equation}\label{EQ:rho-def}
	\rho(x, y) = 2^{-h}
	\qquad\text{if } x \neq y, 
	\quad\text{where } 
	h = \min \{ i \geq 0 \colon x_i \neq y_i \}.
\end{equation}
The pair $(\Sigma,\rho)$ is a compact metric space.
For every word $w\in\AAa^*$, the \emph{cylinder} $[w]$ is the clopen (that is, closed and open) set
$\{x\in\Sigma\colon x_{[0,\abs{w})}=w\}$.

A \emph{shift} is the map
$\sigma\colon \Sigma \to \Sigma $ defined by $\sigma(x_0 x_1 x_2 \ldots) = x_1 x_2 \ldots$.
For each nonempty closed $\sigma$-invariant subset $Y \subseteq \Sigma$,
the restriction of $(\Sigma, \sigma)$ to $Y$ is called a \emph{subshift}.
The closure of the orbit 
$\{ \sigma^n(x) \colon n \geq 0 \}$ of any $x \in \Sigma$ defines a subshift,
as it is always a nonempty, closed and $\sigma$-invariant set.

\subsection{Recurrence plot} \label{SUBS:RecPlot}
A recurrence plot \cite{eckmann1987recurrence} visualizes trajectory of a dynamical system.
For a sequence $x\in\Sigma$, $n \in \NNN \cup \{ \infty \}$, $n \geq 2$ and $\eps >0$, the \emph{recurrence plot} $\recplot(x,n, \eps)$ is a square $n\times n$ matrix
such that, for $0 \leq i, j < n$,
$$ \recplot(x, n, \eps)_{i,j} = 
\begin{cases}
	1 &\text{  if } \rho(\sigma^i(x), \sigma^j(x)) \leq \eps, \\
	0 &\text{  otherwise}.
\end{cases}
$$
Every pair $(i,j)$ with $\recplot(x, n, \eps)_{i,j}=1$ is called a \emph{recurrence}.

Let $\recplot(x, n, \eps)$ be a recurrence plot and $\ell\ge 1$ be an integer.
A \emph{(diagonal) line of length $\ell$} (or, shortly, an \emph{$\ell$-line}) in $\recplot(x, n,\eps)$ is a triple $(i,j,\ell)$ of integers, where
\begin{itemize}
	\item $0 \leq i,j \leq n-\ell$ and $i\ne j$;
	\item $\recplot(x, n, \eps)_{i+h, j+h}  = 1$ for every $0 \leq h < \ell$;
	\item if $\min\{i,j\} > 0$, then $\recplot(x, n, \eps)_{i-1, j-1}  = 0$;
	\item if $\max\{i,j\} < n-\ell$, then $\recplot(x, n, \eps)_{i+\ell, j+\ell}  = 0$.
\end{itemize}
The pair $(i,j)$ is called the \emph{starting point}
and $\ell$ is called the \emph{length} of the line $(i,j,\ell)$.
If $\min\{i, j\} = 0$ we say that the line is \emph{$0$-boundary}.
Similarly,  if $\max\{i,j\} = n -\ell$, 
we say that the line is \emph{$n$-boundary}
(notice that if $n = \infty$ then no line is $n$-boundary; further, for $n$ finite,
a line can be both $0$-boundary and $n$-boundary).
Lines in $\recplot(x, n, \eps)$, which are neither $0$-boundary nor $n$-boundary
are called \emph{inner lines}. 

In $\recplot(x, \infty, \eps)$ we analogously define also \emph{(diagonal) lines of length $\ell=\infty$}.
However, such lines occur only for eventually $\sigma$-periodic $x$, as is shown in the following proposition. 
Recall that 
$x$ is \emph{$\sigma$-periodic} if there is $p\in\NNN$ such that $x=(x_{[0,p)})^\infty$; the 
smallest such $p$ is called the \emph{period} of $x$.
Further, $x$ is \emph{eventually $\sigma$-periodic} if there is $h\in\NNN$ such that $\sigma^h(x)$ is $\sigma$-periodic.

\begin{proposition}\label{PROP:infinite-lines}
Let $x\in\Sigma$ and $\eps\in(0,1)$. Then the following are true:
\begin{enumerate}
	\item\label{IT:infinite-lines:periodic}
	if $x$ is $\sigma$-periodic then $\recplot(x, \infty, \eps)$ has no line of finite length and has infinitely many lines of infinite length;
	\item\label{IT:infinite-lines:evper}
	if $x$ is eventually $\sigma$-periodic then $\recplot(x, \infty, \eps)$ has infinitely many lines of infinite length, and it has either zero, or at least one but finitely many, or infinitely many lines of finite length;
	\item\label{IT:infinite-lines:nonevper}
	if $x$ is not eventually $\sigma$-periodic then $\recplot(x, \infty, \eps)$ has no line of infinite length
	and has inner lines of arbitrarily large finite length.
\end{enumerate}
\end{proposition}
\begin{proof}
\eqref{IT:infinite-lines:periodic} Assume that $x$ is $\sigma$-periodic with period $p$. Clearly, all the lines in $\recplot(x, \infty, \eps)$ are $(0,kp,\infty)$ and $(kp, 0, \infty)$, where $k\in\NNN$.

\eqref{IT:infinite-lines:evper} Assume now that $x$ is eventually $\sigma$-periodic but not $\sigma$-periodic;
let $h\in\NNN$ be the smallest integer such that $y=\sigma^h(x)$ is periodic and let $p$ be the period 
of $y$. Then all the infinite lines in $\recplot(x, \infty, \eps)$ are $(h,kp+h,\infty)$ and $(kp+h, h, \infty)$, where $k\in\NNN$. Sequences $01^\infty$, $001^\infty$ and $0(01)^\infty$ show that $\recplot(x, \infty, \eps=1/2)$ may have zero, finitely many ($(0,1,1)$ and $(1,0,1)$) or infinitely many ($(0,2k+1,1)$ and $(2k+1,0,1)$ for every $k\in\NNN_0$) lines of finite length.

\eqref{IT:infinite-lines:nonevper} Assume that $(i,j,\infty)$ is a line in $\recplot(x, \infty, \eps)$.
We may assume that $i<j$; put $p=j-i$. Since $\eps<1$, $x_{i+h}=x_{j+h}$ for every $h\in\NNN_0$. Thus we have that
$x=x_{[0,i)}(x_{[i,j)})^\infty$ and so $x$ is eventually $\sigma$-periodic.

Let $x$ be not eventually $\sigma$-periodic. We are going to show that the recurrence plot contains inner lines of arbitrarily large finite length. Suppose, on the contrary, that every line in $\recplot(x, \infty, \eps)$ 
has length smaller than $\ell$. Let $g\in\NNN$ be such that $\eps=2^{-g}$. By the pigeonhole principle, there is a word $w$ of length $\ell+g-1$ which occurs infinitely many times in $x$. Hence there are $0<i<j$ such that both 
$\sigma^i(x)$ and $\sigma^j(x)$ starts with $w$. Then 
$\recplot(x, n, \eps)_{i+h, j+h}  = 1$ for every $h\in[0,\ell)$. This clearly implies that there are $\ell'\ge\ell$
and $0\le k\le i$ such that $(i-k,j-k,\ell')$ is a line in $\recplot(x, \infty, \eps)$, a contradiction
with the choice of $\ell$.
\end{proof}

We include here also the following result showing that, for binary alphabet, symbolic recurrence plots
often contain infinitely many lines of length $1$.

\begin{proposition}\label{PROP:1-lines}
Let $x\in\Sigma$ be not eventually $\sigma$-periodic and let $\eps_0=1/2$.
\begin{enumerate}
	\item\label{IT:infinite-lines:binary}
	If $\card\AAa=2$ then $\recplot(x, \infty, \eps_0)$ has infinitely many lines of length $1$.
	
	\item\label{IT:infinite-lines:other}
	If $\card\AAa\ge 3$ then $\recplot(x, \infty, \eps_0)$ has either zero, or at least one but finitely many, or infinitely many lines of length $1$.
\end{enumerate}

\end{proposition}
\begin{proof}
\eqref{IT:infinite-lines:binary} 
Assume first that there is $a\in\AAa$ such that the $2$-word $aa$ occurs in $x$ infinitely many times.
Since $\AAa$ is finite and $x$ does not end with $a^\infty$, we have that:
\begin{itemize}
	\item there is a letter $b\ne a$ such that the word $aab$ is contained in $x$ infinitely many times;
	\item there is a letter $c\ne a$  such that the word $caa$ is contained in $x$ infinitely many times.
\end{itemize}
Hence $\recplot(x, \infty, \eps_0)$ has infinitely many inner lines $(i,j,1)$ of length $1$, 
where $i\in\NNN$ is such that $x_{[i-1,i+2)}=aab$ and $j\in\NNN$ is such that $x_{[j-1,j+2)}=caa$.

Now \eqref{IT:infinite-lines:binary} easily follows. In fact, let $\AAa=\{0,1\}$. 
Since $x$ does not end with $(01)^\infty$, at least one
of the $2$-words $00$, $11$ occurs in $x$ infinitely many times. Hence, as shown above, $\recplot(x, \infty, \eps_0)$ has infinitely many inner lines of length $1$.

\eqref{IT:infinite-lines:other} 
We may assume that $A\supseteq\{0,1,2\}$.
By \eqref{IT:infinite-lines:binary}, to prove \eqref{IT:infinite-lines:other} it suffices to find $\{0,1,2\}$-sequences
$x$ and $x'$ which are not eventually $\sigma$-periodic and are such that  $\recplot(x, \infty, \eps_0)$ does not 
contain any $1$-line and $\recplot(x', \infty, \eps_0)$ contains at least one but finitely many $1$-lines.

Fix any infinite subset $S\subseteq \NNN_0$ containing $0$. Let $y\in\Sigma=A^{\NNN_0}$ be defined by 
$y_i=0$ if $i\in S$ and $y_i=1$ if $i\not\in S$. Define $x\in\Sigma$ by applying the following substitution rule $\zeta$
on $y$ (that is, $x=\zeta(y)$; for details on substitutions, see the next subsection):
\begin{itemize}
\item $\zeta(0)=012$;
\item $\zeta(1)=12$.
\end{itemize}
We may assume that $S$ is such that $x$ is not eventually $\sigma$-periodic.
The sequence $x=0 1 2 \dots$ starts with $0$ and the sets of all words of length $2$ and $3$ contained in $x$ are
$\lang_2(x)=\{01, 12, 20, 21\}$ and
$\lang_3(x)=\{012, 120, 121, 212\}$. Clearly, $\recplot(x, \infty, \eps_0)$ contains
neither inner nor $0$-boundary line of length $1$. 

Now define $x'\in\Sigma$ by $x'_{[0,6)}=120012$ and $x'_{6+i}=x_i$ for every $i\in\NNN_0$.
Then $x'$ is not eventually $\sigma$-periodic and the sets of 
of all words of length $2$ and $3$ contained in $x$ are
$\lang_2(x')=\lang_2(x)\sqcup \{00\}$ and
$\lang_3(x')=\lang_3(x)\sqcup\{001,200\}$.
Clearly, $\recplot(x', \infty, \eps_0)$ does not contain any $0$-boundary $1$-line.
Further, $(i,j,1)$ is an inner line in $\recplot(x', \infty, \eps_0)$
if and only if either $x_{[i-1,i+2)}=001$ and $x_{[j-1,j+2)}=200$, or vice versa.
Since both words $001$ and $200$ are contained in $x$ exactly once, 
we have that $\recplot(x', \infty, \eps_0)$ contains exactly two inner lines of length $1$.
\end{proof}

Thus, for binary alphabet and not eventually $\sigma$-periodic sequence $x$, 
we always have lines of length $\ell=1$ in the symbolic recurrence plot.
However, this is no longer true for any $\ell\ge 2$, even if $x$ is a 
fixed point of a uniform substitutions (for the corresponding definitions, see the next subsection). 
In fact, for every $\ell\ge 2$ there is a uniform substitution $\zeta$ over binary alphabet
such that the symbolic recurrence plot of $x=\zeta^\infty(0)$ contains no line of length $\ell$.
For example, if $\zeta$ is given by $\zeta(0)=0101$ and $\zeta(1)=0010$
(or by $\zeta(0)=011$ and $\zeta(1)=010$) then
there is no line of length $2$ (or of length $3$, respectively)
in the recurrence plot.

\subsection{Substitution}\label{SUBS:substitution}
From now on we restrict our attention to the binary alphabet
$\AAa=\{0,1\}$. Put $\bar{0} = 1$ and $\bar{1} = 0$. 
The set of all words over $\AAa$ is
\begin{equation*}
	\AAa^{*}  = \bigcup_{\ell \ge 0} \AAa^\ell.
\end{equation*}
The set $\AAa^*$ endowed with concatenation is a free monoid.
Any member of $\AAa^\ell$ is called a \emph{word of length $\ell$}, or an \emph{$\ell$-word}, and is denoted by $w=w_0 w_1 \dots w_{\ell-1}$; here $w_i$ is the \emph{$i$-th letter} of $w$.
Length of a word $w$ will be denoted by $\abs{w}$.
The unique word of length $0$ is called the \emph{empty word} and is denoted by $\emptyword$.
A \emph{prefix} (\emph{suffix}) of a word $w=w_0 w_1 \dots w_{\ell-1}$ is any word
$w_{[0,k)}$ ($w_{[\ell-k,\ell)}$, respectively), where $0\le k\le \ell$; if $k<\ell$, the prefix (suffix) is
said to be \emph{proper}.

A map $\zeta\colon \AAa\to\AAa^*$ is said to be a \emph{substitution $\zeta$ of constant length}
or \emph{uniform substitution} if there is an integer $q\ge 2$ such that $\abs{\zeta(a)}=q$ for every $a\in\AAa$.
So $\zeta$ maps each letter ($0$ and $1$) to a word of length $q$ ($\zeta(0)$ and $\zeta(1)$, respectively).

The substitution $\zeta$ induces a morphism (denoted also by $\zeta$) of the monoid $\AAa^*$
by putting $\zeta(\emptyword) = \emptyword$ and $\zeta(w) = \zeta(w_0) \zeta(w_1) \dots \zeta(w_{\ell-1})$
for any nonempty $\ell$-word $w = w_0 w_1 \dots w_{\ell-1}$.
Further, $\zeta$ induces a map
(again denoted by $\zeta$) from $\Sigma=\AAa^{\NNN_0}$ to $\Sigma$ by
$\zeta(x)=(\zeta(x_n))_{n\in\NNN_0}$ for $ x=(x_n)_{n\in\NNN_0} \in \Sigma$.
Since no confusion can arise, all these maps will be simply called \emph{substitution} $\zeta$.
The iterates $\zeta^k (k \geq 0)$ of $\zeta$ are defined inductively: $\zeta^0$ is the identity 
and $\zeta^k = \zeta \circ\zeta^{k-1}$ for $k\ge 1$.

The \emph{language} of the substitution $\zeta$ is
\begin{equation*}
	\lang_\zeta = \{w\in\AAa^*\colon w \text{ is a subword of some } \zeta^k(a)\}.
\end{equation*}
Words from $\lang_\zeta$ are called \emph{allowed}.
The substitution $\zeta$ defines the \emph{substitution dynamical system} (or, shortly, \emph{$\zeta$-shift}), 
which is the subshift $(X_\zeta,\sigma)$ with
\begin{equation*}
	X_\zeta = \{
		y\in\Sigma\colon
		y_{[0,n)} \in\lang_\zeta \text{ for every } n\ge 1
	\}.
\end{equation*}
The substitution $\zeta$ is called \emph{aperiodic} if $X_\zeta$ contains a sequence
which is not $\sigma$-periodic \cite[Definition~5.15]{queffelec2010substitution}.
The substitution $\zeta$ is \emph{primitive} if there exists $k \geq 1$ such that, for every $a,b \in \AAa$,
$b$ is in $\zeta^k (a)$; see e.g.~\cite[Definition~5.3]{queffelec2010substitution}.

Let $\zeta$ be a constant length substitution.
From now on we make the following assumptions:
\begin{flalign}
	\label{assumpt-start-with-0}
	&\zeta(0) \text{ starts with letter } 0;
	\\
	\label{assumpt-both-0-and-1-in-zeta-k}
	&\zeta\text{ is primitive};
	\\
	\label{assumpt-aperiodic}
	&\zeta\text{ is aperiodic.}
\end{flalign}
A direct consequence of assumption~\eqref{assumpt-aperiodic} for binary alphabet is that
\begin{equation}\label{assumpt-one-to-one}
	\zeta\text{ is injective};
\end{equation}
that is, $\zeta$ is one-to-one on the alphabet: $\zeta(0)\ne \zeta(1)$.

The first assumption \eqref{assumpt-start-with-0} guarantees the existence of a unique fixed point $x \in\Sigma$ of substitution $\zeta$ starting with $0$ \cite[page~126]{queffelec2010substitution}; obviously
$x = \lim\limits_{k \to \infty} \zeta^k (0) =\zeta^\infty(0)$; if we write $\zeta(0)=0w$ (where $w\in\AAa^*$) then
\begin{equation*}
	x = 0\zeta(w)\zeta^2(w)\dots\zeta^k(w)\dots
\end{equation*}

The second assumption \eqref{assumpt-both-0-and-1-in-zeta-k} of primitivity
together with \eqref{assumpt-start-with-0} imply
that $X_\zeta$ is equal to the $\sigma$-orbit closure of the sequence $x$
\cite[Proposition~5.5]{queffelec2010substitution};
hence, $\lang_\zeta$ is equal to the set of all subwords $x_{[i,i+\ell)}$ ($i,\ell\ge 0$) of $x$.
By \cite{michel1976stricte} (see also \cite[Theorem~5.6]{queffelec2010substitution}),
the conditions \eqref{assumpt-start-with-0} and \eqref{assumpt-both-0-and-1-in-zeta-k} also imply that
$(X_\zeta, \sigma)$ is strictly ergodic, that is, minimal and uniquely ergodic.

Concerning aperiodicity \eqref{assumpt-aperiodic}, recall that if $\zeta$ is primitive, 
then it is aperiodic if and only if $X_\zeta$ is infinite \cite[page~151]{queffelec2010substitution}. 
In \cite{seebold1988periodicity} (see also \cite[Theorem~0.1]{tan2008periodicity}) the author characterized
binary substitutions such that $x=\zeta^\infty(0)$ is eventually $\sigma$-periodic. 
Using this result we can reformulate conditions \eqref{assumpt-start-with-0}--\eqref{assumpt-aperiodic} as follows:
\begin{itemize}
\item $\zeta(0)$ starts with $0$ and contains $1$;
\item $\zeta(1)\ne\zeta(0)$ and contains $0$;
\item if $q=2s+1$ is odd, then $\zeta(0)\ne (01)^s0$ or $\zeta(1)\ne (10)^s1$.
\end{itemize}

\begin{notation}
Till the end of Section~\ref{S:densities}, $\zeta$ will always denote a $q$-uniform substitution over binary alphabet satisfying \eqref{assumpt-start-with-0}--\eqref{assumpt-aperiodic},
and $x$ will denote the unique fixed point $\zeta^{\infty}(0)$
of $\zeta$ starting with $0$. Further, in the notation for recurrence plot we usually skip $x$; so $\recplot(n,\eps)$ will always mean $\recplot(x,n,\eps)$.
\end{notation}

\subsection{Recognizability of substitutions}\label{SUBS:recognizability}

\begin{remark}\label{REMARK:unique-zeta-at-indices-of-type-qi}
	Notice that $x = \zeta(x) = \zeta(x_0) \zeta(x_1) \zeta(x_2) \dots$.
	Since $\zeta$ is one-to-one on the alphabet \eqref{assumpt-one-to-one}, for every $i \in \NNN_0$ we have a unique $a \in \AAa$ such that $\word{x}{iq}{(i+1)q} = \zeta(a)$.
\end{remark}

Recognizability of $\zeta$ will play a crucial role in our analysis.
In the following we recall the notion of $1$-cutting of a word from \cite[p.~210]{fogg2002substitutions}.
First, the set of \emph{cutting bars of order 1} is defined by
\begin{equation}\label{E1}
	E_1 = \{0\} \ \cup\  \{ \abs{\zeta(\word{x}{0}{k})}\colon k\in\NNN \}. 
\end{equation}

\begin{definition}\label{DEF:1-cutting} 
	Let $w\in\lang_\zeta$ be nonempty and $i\in\NNN_0$ be such that $w=x_{[i,i+\abs{w})}$. 
	A \emph{$1$-cutting} at the index $i$ of $w$ is any $(k+2)$-tuple ($k\ge 0$) of words
	$$
		[s, v_0, v_1, \dots, v_{k-1}, t],
	$$
	such that, for some $j\in\NNN_0$,
	\begin{enumerate}
	\item\label{LAB-1cutting-w} 
		$w=sv_0v_1\dots v_{k-1}t$;
	\item\label{LAB-1cutting-v} 
		$v_h=\zeta(x_{j+h})$ for every $h\in[0,k)$;
	\item\label{LAB-1cutting-st} 
		$s$ is a proper suffix of $\zeta(x_{j-1})$ ($s=\emptyword$ for $j=0$)
		and
		$t$ is a proper prefix of $\zeta(x_{j+k})$;
	\item\label{LAB-1cutting-E} 
		$E_1\cap [i,i+\abs{w})	= (i-l) \ + \ E_1\cap [l,l+\abs{w})$,
		where $l=\abs{\zeta(x_{[0,j)})} - \abs{s}$.
	\end{enumerate}
\end{definition}

Since $x=\zeta(x)$, every sufficiently long allowed word $w$ has a $1$-cutting at every index $i$
such that $w = x_{[i, i + |w|)}$. 
We are interested in words having a unique $1$-cutting. However, due to the fact 
that some short words can have no $1$-cutting, we define recognizable words as follows. 

\begin{definition}\label{DEF:recognizableWord}
	A nonempty word $w\in\lang_\zeta$ is \emph{recognizable} if 
	there is a (unique) integer $p_w \in [0, q)$
	such that $\{i \in \NNN_0 \colon \word{x}{i}{i+\abs{w}} = w \} \subseteq q \NNN_0 + p_w$.
\end{definition}

The connection between recognizability of a word and uniqueness of its $1$-cuttings is described 
in Appendix~\ref{APP:1cutting}. The main fact is that, for $q$-uniform substitutions
satisfying \eqref{assumpt-start-with-0} and \eqref{assumpt-one-to-one},
an allowed word of length at least $q$ is recognizable if and only if it has a unique $1$-cutting;
see Corollary~\ref{COR:recognizable}.

By \cite[Definition~5.14]{queffelec2010substitution} or
\cite[Definition~4.30]{bruin2022topological},
$\zeta$ is said to be \emph{(unilaterally} or \emph{one-sided) recognizable}
if there exists an integer $K > 0$ such that
$n \in E_1$ and $x_{[n, n+K+1)} = x_{[m, m+K+1)}$ imply $m \in E_1$;
the smallest integer $K$ enjoying this property is called the \emph{recognizability index} of $\zeta$.
Thus, a substitution of constant length $q$ is recognizable if and only
if there is $K>0$ such that every word of the form $x_{[hq,hq+K+1)}$ ($h\in\NNN_0$) is recognizable.
This clearly does not mean that every word of length at least $K+1$ is recognizable, see the following example.

\begin{example}\label{EX:recognizability}
	Let $\zeta\colon\AAa\to\AAa^5$ be given by $\zeta(0)=00101$ and $\zeta(1)=00111$. The recognizability index of $\zeta$ is
	$K=1$, since clearly $x_{[i,i+2)}=00$ if and only if $i\in 5\NNN_0 = E_1$. On the other hand, the word $w=010$ of length $3$ is not recognizable since $x_{[1,4)}=x_{[3,6)}=010$.
\end{example}

By Mentzen (1989) and \cite{apparicio1999reconnaissabilite}
(see also \cite[Proposition~5.14]{queffelec2010substitution} or \cite[Theorem~4.31]{bruin2022topological}), 
a primitive, aperiodic and constant length substitution,
which is one-to-one on the alphabet, is recognizable.
This fact immediately implies that there is the smallest integer $K'$ 
such that every word $w\in\lang_\zeta$ of length at least $K'$ is recognizable.
Clearly, $K+1\le K'\le K+q$; 
as Example~\ref{EX:recognizability} shows, it may happen that $K'>K+1$.
So, every substitution satisfying assumptions \eqref{assumpt-start-with-0}--\eqref{assumpt-aperiodic}  has all
sufficiently long words recognizable.

\begin{definition}\label{DEF:alpha-and-beta}
	Let $\alpha, \beta \in \NNN_0$ and $\recog\in\NNN$ be defined as follows:
	\begin{enumerate}[label=(\alph*)]
	\item 
	$\alpha$ is the length of the longest common prefix of $\zeta(0), \zeta(1)$:
	$$ 
		\word{\zeta(0)}{0}{\alpha} = \word{\zeta(1)}{0}{\alpha},
		\quad 
		\zeta(0)_\alpha \neq \zeta(1)_\alpha.
	$$
	
	\item
	$\beta$ is the length of the longest common suffix of $\zeta(0), \zeta(1)$:
	$$ 
		\word{\zeta(0)}{q-\beta}{q} = \word{\zeta(1)}{q-\beta}{q},
		\quad 
		\zeta(0)_{q-\beta-1} \neq \zeta(1)_{q-\beta-1}.
	$$
	
	\item 
	$\recog\in\NNN_0$ is the smallest integer such that $\recog > \alpha+\beta$ and every word $w\in\lang_\zeta$
	of length (at least) $\recog$ is recognizable.
	\end{enumerate}
\end{definition}

Since $\zeta (0) \neq \zeta (1)$ by \eqref{assumpt-one-to-one}, both $\alpha, \beta$ are well-defined and 
$\alpha + \beta \le q-1$.
Words with length at least $\recog$ will be called \emph{$\recog$-recognizable};
by the definition of $\recog$, every $\recog$-recognizable word is recognizable.

\begin{remark}\label{REM:recog-bounds}
	Note that always $\recog \ge 3$ since, by our assumptions, there always exists a non-recognizable word of length $2$.\footnote{For if not and every $2$-word from $\lang_\zeta$ is recognizable, we have that $q\le m_2$, where $m_2$ is the number
	of $2$-words from $\lang_\zeta$. The case $q=m_2$ leads to $\zeta(0)=\zeta(1)$ and $\sigma$-periodic $x$. 
	In the case $q=2$, we need to check the Thue-Morse and period-doubling substitutions;
	this is done in Example~\ref{EXAM:def-recog}.
	Finally, the case $q=3$ and $m_2=4$ easily leads to a contradiction with 
	the assumption that every $2$-word is recognizable.}
	The period-doubling substitution is an example of a substitution with $\recog=3$, see Example~\ref{EXAM:def-recog}.
	On the other hand, using \cite[Theorem~1]{klouda2016synchronizing} (see also \cite[Theorem~5]{durand2017constant})
	one can show that $\recog < q^3$ for every uniform substitution $\zeta$ satisfying
	\eqref{assumpt-start-with-0}--\eqref{assumpt-aperiodic}.\footnote{In the case when $q\ge 5$, 
	the inequality $\recog < q^3$ is a consequence of
	\cite[Theorem~1]{klouda2016synchronizing} and the fact that every word of length at least $2Z_{\min}+q$ is recognizable, where $Z_{\min}$ is the synchronizing delay of $\zeta$
	\cite{klouda2016synchronizing}. The remaining cases follows by considering all uniform substitutions 
	satisfying \eqref{assumpt-start-with-0}--\eqref{assumpt-aperiodic}:
    $\recog\le 4$ for $q=2$, $\recog\le 7$ for $q=3$ and $\recog\le 20$ for $q=4$.}
	Finally, let us note that the assumption $\recog > \alpha+\beta$ will be used in the proof of Lemma~\ref{LMM:possible-lengths}.
\end{remark}

\begin{example}\label{EXAM:def-recog}
For the Thue-Morse substitution ($\zeta(0)=01$, $\zeta(1)=10$) we have $\recog=4$, since $\alpha=\beta=0$,
the words $00$, $11$, $0101$ and $1010$ are recognizable (hence all $4$-words from $\lang_\zeta$ are recognizable), but $010$ is not.
For the period-doubling substitution ($\zeta(0)=01$, $\zeta(1)=00$) we have $\recog=3$ since
$\alpha=1$, $\beta=0$, the words $1$ and $000$ are recognizable (hence all $3$-words from $\lang_\zeta$ are recognizable),
but $00$ is not.
Thus, for these two substitutions one has $\recog > \alpha+\beta+1$.

On the other hand, for the substitution $\zeta$ given by $\zeta(0)=01110$ and $\zeta(1)=01010$,
we have $\recog=\alpha+\beta+1=5$, but even every $4$-word from $\lang_\zeta$ is recognizable.
Indeed, one can check that the words
$00$, all $3$-words but $010$, and $4$-words $0010$ and $0100$ from $\lang_\zeta$ are recognizable. 
\end{example}


\section{Lines in infinite symolic recurrence plot $\recplot (\infty, \eps_0=1/2)$}\label{S:lengths-and-density}

Recall that $\zeta$ is a substitution satisfying assumptions  \eqref{assumpt-start-with-0}--\eqref{assumpt-aperiodic}
and $x = \zeta^\infty(0)$ is the unique fixed point of $\zeta$ starting with $0$.
Till the end of this section, we deal with
the infinite recurrence plot $\recplot(\infty, \eps_0)$ for the sequence $x = \zeta^\infty(0)$,
where $\eps_0=1/2$.

\subsection{Inner lines and inner line-patterns}\label{SUBS:lines-in-rec-plot}
\begin{definition}\label{DEF:line-pattern}
	An \emph{inner line-pattern} is a triple $\linepattern{a}{w}{b}$ such that
	$w$ is a nonempty word, $a, b \in \AAa$ and
	there are (distinct) integers $i, j \in \NNN$ such that
	\begin{equation}\label{EQ:defILP}
		awb = \word{x}{i-1}{i+\abs{w}+1} 
		\quad\text{and}\quad
		\bar{a}w\bar{b} = \word{x}{j-1}{j+\abs{w}+1}.
	\end{equation}
	The word $w$ will be called \emph{generating word} and $a,b\in \AAa$
	will be called \emph{separators} of the inner line-pattern $\linepattern{a}{w}{b}$.
	For an inner line-pattern $\linepattern{a}{w}{b}$,
	the set of all pairs $(i,j)$ satisfying \eqref{EQ:defILP} 
	is denoted by $\indexoflines{w}{a}{b}$.
	We say that a line-pattern $\linepattern{a}{w}{b}$ is \emph{$\recog$-recognizable}
	if $w$ is $\recog$-recognizable.
\end{definition}
Often an inner line-pattern $\linepattern{a}{w}{b}$ will be denoted by a 
capital letter $P$ or $Q$
(possibly with some subscript or superscript).
If $P=\linepattern{a}{w}{b}$, then
$w$, $a$, $b$ and $\indexoflines{w}{a}{b}$ will be denoted
by $w(P)$, $a(P)$, $b(P)$ and $\iii(P)$, respectively.
Moreover, the length of $w$ will be denoted by $\abs{P}$.

Note that line-patterns $\linepattern{a}{w}{b}$ and $\linepattern{\bar{a}}{w}{\bar{b}}$ are 
mirror images of themselves, in the sense that $(i, j) \in \indexoflines{w}{a}{b}$
if and only if $(j, i) \in \indexoflines{w}{\bar{a}}{\bar{b}}$.

Immediately from the definitions one has that inner lines in (infinite) recurrence plot $\recplot(\infty,\eps_0)$ are tightly connected with inner line-patterns.

\begin{proposition}\label{PROP:lines-and-patterns}
	Let $\ell \geq 1$ and $i, j \in\NNN$ be distinct.
	Then, the triple $(i, j, \ell)$ is an inner line in 
	$\recplot(\infty,\eps_0)$ if and only if 
	there is an inner line-pattern $P$ such that
	\begin{equation*}
		(i, j) \in \iii(P)
		\quad\text{and}\quad
		\abs{P}=\ell.
	\end{equation*}
	If this is the case, $P$ is clearly unique.
\end{proposition} 

This result together with Propositions~\ref{PROP:infinite-lines} and \ref{PROP:1-lines} 
and the fact that $x$ is not eventually $\sigma$-periodic (since $X_\zeta$ is infinite and minimal)
yield the following.

\begin{proposition}\label{PROP:existence-of-patterns}
	There are infinitely many pairwise different inner line-patterns. Further, there is at least one inner line-pattern of length $1$.
\end{proposition}

We finish this subsection with the following trivial lemma.

\begin{lemma}\label{LMM:I(P)-disjoint}
	Let $P,Q$ be different inner line-patterns. Then
	$\iii(P)$ and $\iii(Q)$ are disjoint.
\end{lemma}

\subsection{Induced inner line-patterns} \label{SUBS:induces-line-patterns}

Recall the definition of $\alpha$ and $\beta$ from Definition~\ref{DEF:alpha-and-beta}.

\begin{proposition}[Substitution of line-patterns]\label{P:substitution-of-line-patterns}
	Let $\linepattern{c}{v}{d}$ be an inner line-pattern.
	Put
	\begin{equation}\label{EQ:inducedILP}
	\begin{split}
		w &= \word{\zeta(c)}{q-\beta}{q}\zeta(v)\word{\zeta(d)}{0}{\alpha};
		\\
		a &=  \zeta(c)_{q-\beta-1};
		\\
		b &= \zeta(d)_{\alpha}.
	\end{split}
	\end{equation}
	Then $\linepattern{a}{w}{b}$ is an inner line-pattern and
		\begin{enumerate}[label=(\alph*)]
		\item \label{P:substitution-of-line-patterns-a}
		$\abs{w} = q\abs{v} + \alpha + \beta$,
		
		\item \label{P:substitution-of-line-patterns-b}
		$\indexoflines{w}{a}{b} \supseteq q  \indexoflines{v}{c}{d} - \beta$.
	\end{enumerate}
\end{proposition}
	
Notice that $awb$ is a subword of $\zeta(cvd)$.

\begin{proof}
	Put $\ell = \abs{v} $.
	Let $(i, j) \in \indexoflines{v}{c}{d}$; that is, $cvd = \word{x}{i-1}{i+\ell+1}$
	and $\bar{c}v\bar{d} = \word{x}{j-1}{j+\ell+1}$.
	
	Since $x$ is a fixed point of $\zeta$,
	$$\zeta (cvd) = \zeta(c) \zeta(v_0) \zeta(v_1) \ldots \zeta(v_{\ell-1})\zeta(d)
				 =\word{x}{q(i-1)}{q(i+\ell+1)}. $$
	Analogously, 
	$$\zeta (\bar{c}v\bar{d}) = \zeta(\bar{c}) \zeta(v_0) \zeta(v_1) \ldots \zeta(v_{\ell-1})\zeta(\bar{d})
	=\word{x}{q(j-1)}{q(j+\ell+1)}. $$
	Since $c \neq \bar{c}$ and $d\neq \bar{d}$, we have $\zeta(c) \neq \zeta(\bar{c})$ and $\zeta(d) \neq \zeta(\bar{d})$ 
	by \eqref{assumpt-one-to-one}. 
	
	Let $w, a, b$ be as defined above. By the definition of $\beta$, $\zeta(c)_{q-\beta-1} \neq \zeta(\bar{c})_{q-\beta-1} $
	and $\zeta(c)_{q-h-1} = \zeta(\bar{c})_{q-h-1} $ for every $h \in [0,\beta)$.
	Similarly, by the definition of $\alpha$, $\zeta(d)_{\alpha} \neq \zeta(\bar{d})_{\alpha} $
	and  $\zeta(d)_h = \zeta(\bar{d})_h $ for every $h \in [0, \alpha)$. 
	This means that 
	$$ 
		awb = \word{x}{qi-\beta-1}{q(i+\ell)+\alpha+1} 
		\quad\text{and}\quad
		\bar{a}w\bar{b} = \word{x}{qj-\beta-1}{q(j+\ell)+\alpha+1}.
	$$
	Thus $\linepattern{a}{w}{b}$ is an inner line-pattern 
	with the length $\abs{w}=q \abs{v} + \alpha + \beta$ and
	$(qi-\beta, qj-\beta) \in \indexoflines{w}{a}{b}$.  
\end{proof}

Motivated by Proposition~\ref{P:substitution-of-line-patterns}, induced line-patterns are introduced.

\begin{definition}\label{DEF:induced}
	Let $Q = \linepattern{c}{v}{d}$, $P = \linepattern{a}{w}{b}$ be inner line-patterns.
	We say that $P$ is \emph{induced} by $Q$ (or that $Q$ \emph{induces} $P$)
	if \eqref{EQ:inducedILP} holds.
\end{definition}

\begin{lemma}[On induced separators]\label{LMM:on-induced-separators}
	The substitution $\zeta$ satisfies exactly one of the following four conditions:
	\begin{enumerate}
		\item[(++)] $a = c$ and $b=d$ whenever $\linepattern{c}{v}{d}$ induces $\linepattern{a}{w}{b}$;
		\item[(+\textendash)] $a = c$ and $b=\bar{d}$ whenever $\linepattern{c}{v}{d}$ induces $\linepattern{a}{w}{b}$;
		\item[(\textendash+)] $a = \bar{c}$ and $b=d$ whenever $\linepattern{c}{v}{d}$ induces $\linepattern{a}{w}{b}$;
		\item[(\textendash \textendash)] $a = \bar{c}$ and $b=\bar{d}$ whenever $\linepattern{c}{v}{d}$ induces $\linepattern{a}{w}{b}$.
	\end{enumerate}
\end{lemma}

\begin{proof}
	By Definition~\ref{DEF:alpha-and-beta}, if $\zeta(c)_{q-\beta-1}=a$ and
	$\zeta(d)_\alpha = b$, then
	$\zeta(\bar c)_{q-\beta-1}=\bar a$ and $\zeta(\bar d)_\alpha = \bar b$.
	From this the lemma immediately follows; for example,
	the case (\textendash+) happens when $\zeta(0)_{q-\beta-1} = 1$ and $\zeta(0)_\alpha = 0$,
	and the case (++) happens when $\zeta(0)_{q-\beta-1} = 0$ and $\zeta(0)_\alpha = 0$.
\end{proof}

\subsection{Induced $\recog$-recognizable inner line-patterns} \label{SUBS:induced-Rrecognizable-line-patterns}
	
\begin{lemma}\label{LMM:possible-lengths}
	Let $P$ be a $\recog$-recognizable inner line-pattern.
	Then $\abs{P} \in q \NNN + (\alpha + \beta)$
	and $\iii(P) \subseteq q (\NNN \times \NNN) - \beta$.
\end{lemma}

\begin{proof}
 	Fix any $P = \linepattern{a}{w}{b}$ and  $(i, j) \in \indexoflines{w}{a}{b}$.
 	Put $\abs{w} = \ell$ and
 	let $p = p_w$ be from Definition~\ref{DEF:recognizableWord}.
 	Then there are $i', j' \in \NNN_0$ such that
 	$i = qi'+p$, $j=qj'+p$ (see an illustration in Tables~\ref{table:proof3.4a} and 			\ref{table:proof3.4b}; there, the vertical bars denote the cutting bars from $E_1=q\NNN_0$,
 	see \eqref{E1}).

\begin{table}[h!!!]
	\begin{center}
		\caption{The first cutting bars of $aw$ and $\bar{a}w$ for $p \geq 1$}
		\label{table:proof3.4a}
		\begin{tabular}{l c c c c c c c c }
			beginning of $aw$ & $a$ & $w_0$ & $w_1$ & $\dots$ & $w_{q-p-1}$ & $|$ &
			 $w_{q-p}$ & $\dots$    \\
			\footnotesize{indexes} & \footnotesize{$i-1$} & \footnotesize{$i$} & \footnotesize{$i+1$} & \footnotesize{$\dots$} & \footnotesize{$i+q-p-1$} &  &
			\footnotesize{$i+q-p = q(i'+1)$} & \footnotesize{$\dots$}
			\\
			beginning of $\bar{a}w$ & $\bar{a}$ & $w_0$ & $w_1$ & $\dots$ & $w_{q-p-1}$ & $|$ &
			$w_{q-p}$ & $\dots$    \\
			\footnotesize{indexes} & \footnotesize{$j-1$} & \footnotesize{$j$} & \footnotesize{$j+1$} & \footnotesize{$\dots$} & \footnotesize{$j+q-p-1$} &  &
			\footnotesize{$j+q-p = q(j'+1)$} & \footnotesize{$\dots$}\\
			\end{tabular}
\end{center}
\end{table}

\begin{table}[h!!!]
	\begin{center}
		\caption{The first cutting bars of $aw$ and $\bar{a}w$ for $p = 0$}
			\label{table:proof3.4b}
			\begin{tabular}{l c c c c c  }
			beginning of $aw$ & $a$ & $|$ & $w_0$ & $w_1$ & $\dots$ \\
			\footnotesize{indexes} & \footnotesize{$i-1$} & & \footnotesize{$i = i'q$} & \footnotesize{$i+1$} & \footnotesize{$\dots$}\\
			beginning of $\bar{a}w$ & $\bar{a}$ & $|$ & $w_0$ & $w_1$ & $\dots$  \\
			\footnotesize{indexes} & \footnotesize{$j-1$} & & \footnotesize{$j = j'q$} & \footnotesize{$j+1$} & \footnotesize{$\dots$}\\
		\end{tabular}
	\end{center}
\end{table}
	
Note that $a = x_{i-1}$, $\bar{a} = x_{j-1}$.
If $p=0$, then $\zeta(0)$ ends with $a$ and $\zeta(1)$ ends with $\bar{a}$, or vice versa;
so $\beta = 0$.
Otherwise, $\zeta(0)$ ends with $a\word{w}{0}{q-p}$
and $\zeta(1)$ ends with $\bar{a}\word{w}{0}{q-p}$, or vice versa; so $\beta = q-p$.
In both cases, $i, j \in q \NNN - \beta$.
We showed that $\indexoflines{w}{a}{b} \subseteq q (\NNN \times \NNN) - \beta$.

It remains to prove that $\ell \in q \NNN + (\alpha + \beta)$.
By recognizability of $w$,
there are integers $r \in [0, q)$ and $i'', j'' \in \NNN_0$ such that
$i + \ell = qi''+r$, $j+\ell = qj''+r$ 
(see an illustration in Tables~\ref{table:proof3.4c}
and \ref{table:proof3.4d}).
As above, using that $b = x_{i + \ell}$ and $\bar{b} = x_{j+\ell}$,
we obtain that $\alpha = r$.
Hence,
$$ 
	\ell = 
	\begin{cases}
		q(i''-i'-1) + \alpha + \beta &\text{  if } p \geq 1, \\
		q(i''-i') + \alpha + \beta &\text{  if } p = 0.
	\end{cases}
$$
Since $\ell \geq \recog > \alpha + \beta$ by Definition~\ref{DEF:alpha-and-beta}, the proof is finished.

\begin{table}[h!!!]
	\begin{center}
		\caption{The last cutting  bars of $wb$ and $w\bar{b}$
			for $r \geq 1$}
		\label{table:proof3.4c}
		\begin{tabular}{l c c c c c c c c c}
			end of $wb$ & $\dots$ & $w_{\ell-r-1}$ & $|$ & $w_{\ell-r}$ & $\dots$ & $w_{\ell-2}$ & $w_{\ell-1}$ & $b$ &\\
			\footnotesize{indexes} & \footnotesize{$\dots$} & \footnotesize{$i+\ell-r-1$} & & \footnotesize{$i+\ell-r = qi''$} & \footnotesize{$\dots$} & \footnotesize{$i+\ell-2$} & \footnotesize{$i+\ell-1$} & \footnotesize{$i+\ell$} &\\
			end of $w\bar{b}$ & $\dots$ & $w_{\ell-r-1}$ & $|$ & $w_{\ell-r}$ & $\dots$ & $w_{\ell-2}$ & $w_{\ell-1}$ & $\bar{b}$ &\\
			\footnotesize{indexes} & \footnotesize{$\dots$} & \footnotesize{$j+\ell-r-1$} & & \footnotesize{$j+\ell-r = qj''$} & \footnotesize{$\dots$} & \footnotesize{$j+\ell-2$} & \footnotesize{$j+\ell-1$} & \footnotesize{$j+\ell$} &\\
		\end{tabular}
	\end{center}
\end{table}

\begin{table}[h!!!]
	\begin{center}
		\caption{The last cutting  bars of $wb$ and $w\bar{b}$
			for $r=0$}
		\label{table:proof3.4d}
		\begin{tabular}{l c c c c } 
			end of $wb$ & $\dots$ & $w_{\ell-1}$ & $|$ & $b$   \\
			\footnotesize{indexes} & \footnotesize{$\dots$} & \footnotesize{$i+\ell-1$} &   & \footnotesize{$i+\ell = qi''$ }  \\
			end of $w\bar{b}$ & $\dots$ & $w_{\ell-1}$ & $|$ & $\bar{b}$   \\
			\footnotesize{indexes} & \footnotesize{$\dots$} & \footnotesize{$j+\ell-1$} &   & \footnotesize{$j+\ell = qj''$}  \\
		\end{tabular}
	\end{center}
\end{table}	
\end{proof}

\begin{lemma}\label{LMM:doplnenie-na-desubstituciu}
	Let $\linepattern{a}{w}{b}$ be a $\recog$-recognizable inner line-pattern; put $\ell=\abs{w}$.
	Then there exist a unique word $v$ and unique letters $c,d$ such that
	$$awb \ \text{ is a subword of }\ \zeta(cvd),$$
	where $\abs{v} = \big({\ell-(\alpha+\beta)}\big)/{q}$.
	Moreover, 
	\begin{equation*}\label{only-indices-in-cor-3.5}
		\{ i \in \NNN\colon \word{x}{i-1}{i+\ell+1} = awb\} 
		\quad\subseteq\quad
		q \{ i' \in \NNN \colon \word{x}{i'-1}{i'+\abs{v}+1} = cvd\} - \beta ,
	\end{equation*}
	and
	\begin{equation}\label{words-in-cor-3.5}
		\begin{split}
		 \word{\zeta(c)}{q-\beta-1}{q} &= a \word{w}{0}{\beta};
		\\
		\zeta(v) &= \word{w}{\beta}{\ell - \alpha};
		\\
		\word{\zeta(d)}{0}{\alpha+1} &= \word{w}{\ell-\alpha}{\ell}b.
		\end{split}
	\end{equation}
\end{lemma}
	
Notice that \eqref{words-in-cor-3.5} is equivalent to \eqref{EQ:inducedILP}.
	
\begin{proof}
	By Lemma~\ref{LMM:possible-lengths},
	$\ell = q\ell' + \alpha + \beta$ for some $\ell' \in \NNN$.
	Let $(i, j) \in \indexoflines{w}{a}{b}$.
	That is, $awb$ starts at the index $i-1$ and $\bar{a}w\bar{b}$ starts at the index $j-1$.
	Moreover, by Lemma~\ref{LMM:possible-lengths},
	$i = qi'-\beta$ and $j = qj'-\beta$ for some $i', j' \in \NNN$.

	This means that $\word{x}{qi'-\beta-1}{qi'} = a \word{w}{0}{\beta}$. 
	By Remark~\ref{REMARK:unique-zeta-at-indices-of-type-qi}, there is unique
	$c \in \AAa$ such that $\word{x}{q(i'-1)}{qi'} = \zeta(c)$.
	Notice that $c$ depends only on $a$ and does not depend on $i'$;
	indeed, the definition of $\beta$ and the fact that $x_{qi'-\beta-1} = a$ uniquely determine whether $c = 0$ or $c=1$, independently on $i'$.
	
	Since $\ell - \alpha - \beta = \ell' q$,
	we have $\word{x}{qi'}{qi'+q\ell'}= \word{w}{\beta}{\ell - \alpha}$.
	By Remark~\ref{REMARK:unique-zeta-at-indices-of-type-qi}, there is a unique word
	$v = v_0 v_1 \dots v_{\ell'-1}$ such that $\zeta(v) = \word{w}{\beta}{\ell - \alpha}$.
	Again, $v$ depends only on $w$ and not on $i'$.
		
	Analogously, from the definition of $\alpha$ and the fact that 
	$\word{x}{qi'+q\ell'}{qi'+q\ell'+\alpha+1}= \word{w}{\ell-\alpha}{\ell}b$, 
	we obtain a unique
	$d \in \AAa$ such that $\word{x}{qi'+q\ell'}{q(i'+1)+q\ell'} = \zeta(d)$.
	Here, $x_{qi'+q\ell'+\alpha+1} = b$ determines whether $d=0$ or $d=1$,
	independently on $i'$ and $\ell'$.
	
	We have showed \eqref{words-in-cor-3.5} and
	$$\{ i \in \NNN\colon \word{x}{i-1}{i+\ell+1} = awb\} \subseteq
	q \{ i' \in \NNN\colon \word{x}{i'-1}{i'+\ell'+1} = cvd\} - \beta,$$
	so the proof is finished.
\end{proof}

\begin{proof}[Proof of Proposition \ref{P:desubstitution-of-line-patterns}]
	Let $P = \linepattern{a}{w}{b}$; put $\ell = \abs{w}$. 
	Let $(i, j) \in \indexoflines{w}{a}{b}$; that is,
	$awb$ starts at the index $i-1$ and $\bar{a}w\bar{b}$ starts at the index $j-1$.

	By Lemma~\ref{LMM:doplnenie-na-desubstituciu} applied to $\linepattern{a}{w}{b}$, we have a unique word $cvd$ (not depending on $i$) such that
	$	\word{\zeta(c)}{q-\beta-1}{q} = a \word{w}{0}{\beta},
	\zeta(v) = \word{w}{\beta}{\ell - \alpha}$ and 
	$\word{\zeta(d)}{0}{\alpha+1} = \word{w}{\ell-\alpha}{\ell}b$.
	Moreover, $\abs{v} = (\ell-(\alpha+\beta))/q$ and $cvd$ starts at $i'-1$, where $i' = (i+\beta)/q$.
	
	Analogously, 	by Lemma~\ref{LMM:doplnenie-na-desubstituciu} applied to $\linepattern{\bar{a}}{w}{\bar{b}}$, we have a unique $euf$ such that
	$	\word{\zeta(e)}{q-\beta-1}{q} = \bar{a} \word{w}{0}{\beta},
	\zeta(u) = \word{w}{\beta}{\ell - \alpha}$ and 
	$\word{\zeta(f)}{0}{\alpha+1} = \word{w}{\ell-\alpha}{\ell}\bar{b}$; moreover
	$\abs{u} = (\ell-(\alpha+\beta))/q$ and $euf$ starts at $j'-1$, where $j' = (j+\beta)/q$.
	
	Trivially, $u=v$ by \eqref{assumpt-one-to-one} and \eqref{words-in-cor-3.5}.
	Since $\zeta(c)_{q-\beta-1} = a$ and  $\zeta(e)_{q-\beta-1} = \bar{a}$, we have $c \neq e$ and so $e = \bar{c}$.
	Similarly, $\zeta(d)_{\alpha} = b$ and $\zeta(f)_{\alpha} = \bar{b}$, so $d \neq f$ and thus $f = \bar{d}$.
	
	We have proved that $\linepattern{c}{v}{d}$ is an inner line-pattern and
	that $(i', j') \in \indexoflines{v}{c}{d}$.
	This shows that  $\indexoflines{w}{a}{b} \subseteq q  \indexoflines{v}{c}{d} -\beta$.
	Since the converse inclusion follows from Proposition~\ref{P:substitution-of-line-patterns},
	$\abs{w} = q\abs{v}+\alpha+\beta$ is trivial and uniqueness of $\linepattern{c}{v}{d}$
	follows from Lemma~\ref{LMM:doplnenie-na-desubstituciu},
	the proof is finished.
\end{proof}

Repeated use of Proposition~\ref{P:desubstitution-of-line-patterns} gives the following corollary.

\begin{corollary}\label{C:repeat-inducing}
Let $P$ be a $\recog$-recognizable inner line-pattern. Then there exist a unique integer
$k=k(P)\ge 1$ (called the \emph{order} of $P$) and unique inner line-patterns $P_0,\dots,P_k$ such that
\begin{enumerate}[label=(\alph*)]
	\item $P_k=P$;
	\item $P_i$ is $\recog$-recognizable and is induced by $P_{i-1}$ for every $i>0$;
	\item $P_0$ is not $\recog$-recognizable.
\end{enumerate}
Moreover, $k$ and $\abs{P_0}$ depends only on the length $\abs{P}$ of $P$,
\begin{equation} \label{EQ:repeated-usage}
	\iii(P) = q^k \iii(P_0) - \beta\frac{q^k-1}{q-1}
	\quad\text{and}\quad
	\abs{P} = q^k\abs{P_0} + (\alpha+\beta) \frac{q^k-1}{q-1}\,.
\end{equation}
\end{corollary}

\begin{lemma}\label{LMM:diff-line-patterns-induces-different-line-patterns}
	Let $P', Q'$ and $P, Q$ be inner line-patterns.
	Assume that $P'$ induces $P$ and $Q'$ induces $Q$.
	Then $P = Q$ if and only if $P' = Q'$.
\end{lemma}

\begin{proof}
	If $P' = Q'$, then $P=Q$ by Definition~\ref{DEF:induced}.
	On the other hand, let $P = Q = \linepattern{a}{w}{b}$.
	Assume that $P' =\linepattern{c}{v}{d}$, $Q' = \linepattern{e}{u}{f}$  both induces $P = Q$.
	By Definition~\ref{DEF:induced} and \eqref{assumpt-one-to-one}, $e=c, f=d$ and $v=u$, hence $P' = Q'$.
\end{proof}

Corollary~\ref{C:repeat-inducing} and Lemmas~\ref{LMM:diff-line-patterns-induces-different-line-patterns} and \ref{LMM:I(P)-disjoint} immediately imply the following.

\begin{lemma}\label{LMM:disjoint-snakes}
	Let $P$ and $Q$ be different $\recog$-recognizable inner line-patterns of equal lengths $\ell$.
	Let $k$ be the order of both $P$ and $Q$, and $P_i, Q_i$ $(0 \leq i \leq k)$
	be the inner line-patterns from Corollary~\ref{C:repeat-inducing} applied to $P, Q$,
	respectively. Then
	$$ 
		P_i \neq Q_i \quad \text{and} \quad \iii(P_i) \cap \iii(Q_i) = \emptyset 
	$$
	for every $ 0 \leq i \leq k$.
\end{lemma} 

Put
\begin{equation}\label{EQ:Kell}
	\kkk_\ell
	=
    \{ (i,j) \in \NNN^2\colon 
    	(i,j,\ell) 
		\text{ is an inner line in } \recplot(\infty, \eps_0) 
	\}.
\end{equation}

\begin{corollary}\label{COR:lenghts}
	Let $\ell \in \NNN$ be such that $\kkk_\ell \neq \emptyset$.
	Then there are $r \in \NNN$ and inner line-patterns $P^1, \dots, P^r$ of length $\ell$
	such that
	\begin{equation}\label{EQ:K_ell id disjoint union}
		\kkk_\ell = \bigsqcup_{j=1}^r   \iii (P^j).
	\end{equation}
	Further, if $\ell \geq \recog$, $k$ is the order of every $P^j$, 
	and $P_i^j$ $(1 \leq j \leq r, \ 0 \leq i \le k )$ are inner line-patterns from Corollary~\ref{C:repeat-inducing} 
	applied to $P^j$, then
	$$ 
		\kkk_\ell = q^k \kkk_{\ell_0} - \beta \frac{q^k-1}{q-1},
		\qquad
		\kkk_{\ell_0} = \bigsqcup_{j=1}^{r} \iii (P_0^j)
	$$ 
	and
	$$ 
		\ell = q^k \ell_0 + (\alpha + \beta)\frac{q^k-1}{q-1},
	$$
	where $\ell_0$ is the length of (every) $P_0^j$.
\end{corollary}

\begin{proof}
	Let $P^1, \dots, P^r $ denote all inner line-patterns of length $\ell$.
	By Proposition~\ref{PROP:lines-and-patterns} and the assumption that $\kkk_\ell$ is nonempty,
	$r \geq 1$ and $\kkk_\ell$ is the union of all $\iii(P^j)$; 
	further, the union is disjoint by Lemma~\ref{LMM:I(P)-disjoint}.
	The rest follows from Corollary~\ref{C:repeat-inducing} and
	Lemma~\ref{LMM:disjoint-snakes}.
\end{proof}

\subsection{Examples}\label{S:examples}

In this subsection, we provide examples of line-patterns and induced line-patterns. Additionally, we present an example that highlights the significant role of recognizability in Lemma~\ref{LMM:doplnenie-na-desubstituciu} and demonstrates that the inclusion in Proposition~\ref{P:substitution-of-line-patterns}(b) cannot be replaced with equality.

\begin{example}[Period-doubling substitution]
	Let $\zeta$ be the period-doubling substitution, that is $\zeta(0) = 01$ and $\zeta(1) = 00$.
	Hence, $q=2$, $\alpha =1$, $\beta = 0$ and $\recog=3$, see Example~\ref{EXAM:def-recog};
	$\zeta$ is of type $(\textendash \, \textendash)$.
	Inner line-patterns of length $1$ and $2$ are  $\linepattern{0}{0}{0}$ and $\linepattern{0}{00}{1}$
	and their mirror images.
	Every longer inner line-pattern is induced by them, as is also described in \cite{vspitalsky2018recurrence}.
\end{example}

\begin{example}[Thue-Morse substitution] \label{example:thue-morse}
	Let $\zeta$ be the Thue-Morse substitution, that is $\zeta(0) = 01$ and $\zeta(1) = 10$.
	Hence, $q=2$, $\alpha = \beta = 0$ and $\recog=4$, see Example~\ref{EXAM:def-recog}.
	Further, $\zeta$ is of type $(\textendash \, +)$ and
	$$ 
		x= 0110\,\,1001\,\,1001\,\,0110\,\,1001\,\,0110\,\,0110\,\,1001\,\dots.
	$$
	Inner lines of length $1$ are determined by patterns $\linepattern{0}{0}{1}$, $\linepattern{0}{1}{1}$ and their mirror images.
	Inner lines of length $2$ are determined by patterns $\linepattern{0}{01}{1}$, $\linepattern{0}{10}{1}$ 
	and their mirror images (these patterns are not induced by $1$-patterns), 
	as well as by patterns $\linepattern{a}{01}{a}$ and $\linepattern{a}{10}{a}$ ($a\in\AAa$) 
	that are induced by 1-patterns.
	Inner lines of length $3$ are determined by patterns $\linepattern{0}{010}{1}$, $\linepattern{0}{101}{1}$ and their mirror images.
\end{example}

\begin{example}
	We show that the assumption of recognizability in Lemma~\ref{LMM:doplnenie-na-desubstituciu} cannot be omitted
	and that, in Proposition~\ref{P:substitution-of-line-patterns}(b), inclusion cannot be replaced by equality.
	Let $\zeta$ be a substitution given by $0 \to 0101$ and $1 \to 0100$.
	Hence, $q=4$, $\alpha =3$, $\beta = 0$ and
	$$ 
		x = 0101\,\, 0100\,\, 0101\,\, 0100\,\, 0101\,\, 0100\,\, 0101\,\, 0101\,\, 0101 \,\,0100 \,\,0101 \,\,0100 \,\,0101 
	\dots . 
	$$
	The word $w=0101010$ is not recognizable, as $w = x_{[i, i+\abs{w})} = x_{[j, j+\abs{w})}$
	for $i=24$ and $j=26$; thus $\recog\ge 8$ (in fact, $\recog=14$).
	
	Clearly, an inner line-pattern $\linepattern{c}{v}{d}=\linepattern{0}{0}{0}$ induces (not $\recog$-recognizable)
	$P=\linepattern{a}{w}{b}$ ($w=0101010$, $a=b=1$) by Proposition~\ref{P:substitution-of-line-patterns}.
	Put $e=u=0$ and $f=1$; then 
	\begin{eqnarray*}
		\zeta(cvd) &= \ 0101\,\,0101\,\,0101 &= \ 0awb01,
	\\
		\zeta(euf) &=  \ 0101\,\,0101\,\,0100 &= \ 0awb00.
	\end{eqnarray*}
	This shows that, in Lemma~\ref{LMM:doplnenie-na-desubstituciu}, the assumption of recognizability cannot be omitted.		
	Further, $(26, 8) \in \indexoflines{w}{a}{b}\setminus(q\indexoflines{v}{c}{d} - \beta)$	
	and so, in Proposition~\ref{P:substitution-of-line-patterns}(b), inclusion cannot be replaced by equality.	
\end{example}

\subsection{Density of (the set of starting points of) inner lines}
\label{subs:density-of-inner-lines}

In this subsection we prove that the density of $\kkk_\ell$ in $\NNN_0^2$ 
$$ 
	\dens(\kkk_\ell)
	= \lim \limits_{n \to \infty} \frac{1}{n^2}
		\card \big(\kkk_\ell \cap [0, n)^2\big)
$$
always exists and is positive if and only if $\kkk_\ell\ne\emptyset$,
and we give a formula for calculating the density of $\kkk_\ell$
with any $\ell\ge\recog$ from that of $\kkk_{\ell_0}$ with $\ell_0<\recog$;
see Theorem~\ref{THM:density}. To this end, we need some auxiliary results.
Let $\mu$ denote a unique invariant measure of $(X_\zeta, \sigma)$.

\begin{lemma} \label{LMM:density-of-line-patterns-is-positive}
	Let $\linepattern{a}{w}{b}$ be an inner line-pattern.
	Then the density of $\indexoflines{w}{a}{b}$ exists and 
	$$\dens(\indexoflines{w}{a}{b}) = \mu([awb])\cdot \mu([\bar{a}w\bar{b}]) > 0.$$
\end{lemma}

\begin{proof}
	Since $\linepattern{a}{w}{b}$ is an inner line-pattern,
	$awb$ and $\bar{a}w\bar{b}$ are subwords of $x$.
	Further, since $(X_\zeta, \sigma)$ is strictly ergodic, the measure of every nonempty 
	cylinder is positive.
	Therefore $\mu([awb])$ and $\mu([\bar{a}w\bar{b}])$ are both positive.
	Using \cite[Corollary~5.3]{queffelec2010substitution} and the equality
	$$ 
		\indexoflines{w}{a}{b} 
		=
		\{ i \in \NNN_0 \colon  \sigma^i (x) \in [awb]\} 
		\times
		\{ j \in \NNN_0 \colon  \sigma^j (x) \in [\bar{a}w\bar{b}] \}
		+1,
	$$
	we can easily obtain that
	$\dens(\indexoflines{w}{a}{b}) = \mu([awb])\cdot \mu([\bar{a}w\bar{b}])$.
\end{proof}

\begin{lemma}\label{LMM:about-densities}
	Let $P$ be a $\recog$-recognizable inner line-pattern and $P_k=P$, $k$, $P_0$ 
	be from Corollary \ref{C:repeat-inducing}. Then
	$$
		\dens (\iii(P_k)) =
		q^{-2k} \dens ( \iii(P_0)).
	$$ 
\end{lemma}
\begin{proof}
	The lemma follows from \eqref{EQ:repeated-usage}
	and \eqref{EQ:density-props}.
\end{proof}

\begin{lemma}\label{P:density-is-sum}
	Let $\ell \in \NNN$ be such that $\kkk_\ell\ne\emptyset$.
	Let $P^j$ $(j=1,\dots,r)$ be inner line-patterns of length $\ell$
	from Corollary~\ref{COR:lenghts}.
	Then the density of $\kkk_\ell$ exists and
	\begin{equation}\label{density-is-sum}
		 \dens(\kkk_\ell) = \sum_{j=1}^r \dens (\iii (P^j)) > 0.
	\end{equation}
\end{lemma}

\begin{proof}
	Since the union in \eqref{EQ:K_ell id disjoint union} is disjoint and finite, the result
	follows from \eqref{EQ:density-props} and Lemma~\ref{LMM:density-of-line-patterns-is-positive}.
\end{proof}

Now we are able to prove Theorem \ref{THM:density}.

\begin{proof}[Proof of Theorem \ref{THM:density}]
	By Lemma~\ref{LMM:density-of-line-patterns-is-positive} and Corollary~\ref{COR:lenghts}
	it suffices to prove the formulas from \eqref{Case3-in-THM:density}.
	Corollary~\ref{COR:lenghts} gives that
	$q^k = (\ell+ c)/(\ell_0 + c)$ with  $c = (\alpha+\beta)/(q-1)$,
	and
	$\kkk_\ell = q^k \kkk_{\ell_0} - \beta (q^k-1)/(q-1)$,
	where $k$ is the order of every inner line-pattern of length $\ell$.
	Since $\dens (\kkk_\ell)= q^{-2k}\dens(\kkk_{\ell_0})$ by Lemma~\ref{LMM:about-densities},
	the formulas follow.
\end{proof}

\subsection{$0$-boundary lines and line-patterns} \label{SUBS:0-boundary}
The results in Subsections~\ref{SUBS:induces-line-patterns}-\ref{SUBS:induced-Rrecognizable-line-patterns}
deal with \emph{inner} line-patterns.
In this subsection we give corresponding results for \emph{$0$-boundary} line-patterns.
The proofs are analogous and are omitted.

\begin{definition}\label{DEF:0-boundary-line-pattern}
	A \emph{$0$-boundary line-pattern} is a pair $w^b$ such that
	$w$ is a nonempty word, $b\in \AAa$ and
	there is an integer $i \in \NNN$ such that
	$$ 
		wb = \word{x}{i}{i + \abs{w} +1}  
		\quad\text{and}\quad
		w\bar{b} = \word{x}{0}{\abs{w}+1}.
	$$
	The set of all such indexes $i$ is denoted by $\indexoflines{w}{}{b}$.
	We say that a line-pattern $w^b$ is \emph{$\recog$-recognizable} 
	if $w$ is $\recog$-recognizable.
\end{definition}

The set $\indexoflines{w}{}{b}$ is tightly connected with $0$-boundary 
lines in infinite recurrence plot $\recplot(\infty,\eps_0)$.
\begin{proposition}
	Let $\ell \geq 1$ be an integer and $i,j\in\NNN_0$ be distinct.
	Then the triple $(i,j,\ell)$ is a $0$-boundary line in $\recplot(\infty,\eps_0)$
	if and only if there is a $0$-boundary line-pattern $w^b$ of length $\ell$
	such that either $i \in \indexoflines{w}{}{b}$ and $j=0$, or vice versa.
\end{proposition} 

As in Subsection~\ref{SUBS:induces-line-patterns},
every $0$-boundary line-pattern induces another such pattern.

\begin{proposition}
	\label{P:substitution-of-line-patterns-for-0-boundary}
	Let $\linepattern{}{v}{d}$ be a $0$-boundary line-pattern.
	Put
	\begin{equation}\label{EQ:induced0LP}
	\begin{split}
		w &= \zeta(v)\word{\zeta(d)}{0}{\alpha};
		\\
		b &=  \zeta(d)_{\alpha}.
	\end{split}
	\end{equation}
	Then $\linepattern{}{w}{b}$ is a $0$-boundary line-pattern and
	\begin{enumerate}[label=(\alph*)]
		\item 
		$\abs{w} = q\abs{v} + \alpha$;
		\item 
		$\iii_w^b \supseteq q  \iii_v^d$.
	\end{enumerate}
\end{proposition}

\begin{definition}\label{DEF:induced-for-0-boundary}
	Let $Q = \linepattern{}{v}{d}$, $P = \linepattern{}{w}{b}$ be $0$-boundary line-patterns.
	We say that $P$ is \emph{induced} by $Q$ (or that $Q$ \emph{induces} $P$)
	if \eqref{EQ:induced0LP} holds.
\end{definition}

\begin{lemma}[On induced separators for $0$-boundary line-patterns]
	\label{LMM:on-induced-separators-on-0-boundary}
	The substitution $\zeta$ satisfies exactly one of the following two conditions:
	\begin{enumerate}
		\item[(+)] $b=d$ whenever $\linepattern{}{v}{d}$ induces $\linepattern{}{w}{b}$;
		\item[(\textendash)] $b=\bar{d}$ whenever $\linepattern{}{v}{d}$ induces $\linepattern{}{w}{b}$.
	\end{enumerate}
\end{lemma}
Clearly, the case (+) in the preceding lemma corresponds to the cases (++) and (\textendash+),
and the case (\textendash) corresponds to the cases (+\textendash) and (\textendash\textendash)
from Lemma~\ref{LMM:on-induced-separators}

Lemma~\ref{LMM:possible-lengths},
Proposition~\ref{P:desubstitution-of-line-patterns},
and Lemmas~\ref{LMM:diff-line-patterns-induces-different-line-patterns} and \ref{LMM:disjoint-snakes}
hold as stated also for $0$-boundary line-patterns,
but with $\beta$ replaced by $0$.

Analogously to \eqref{EQ:Kell} we define
\begin{equation}
	\kkk_\ell^0 
	= 
	\{ (i,j) \in \NNN_0^2\colon 
		(i,j,\ell) \text{ is a } 0 \text{-boundary-line in } 
		\recplot(\infty, \eps_0) 
	\}.
\end{equation}

\begin{corollary}\label{C:repeat-inducing-0-boundary}
	Let $P$ be a $\recog$-recognizable $0$-boundary line-pattern. Then there exist 
	a unique integer $k=k(P)\ge 1$ (called the \emph{order} of $P$)
	and unique $0$-boundary line-patterns $P_0,\dots,P_k$ such that
	\begin{enumerate}[label=(\alph*)]
		\item $P_k=P$;
		\item $P_i$ is $\recog$-recognizable and is induced by $P_{i-1}$ for every $i>0$;
		\item $P_0$ is not $\recog$-recognizable.
	\end{enumerate}
	Moreover, $k$ and $\abs{P_0}$ depends only on the length $\abs{P}$ of $P$,
	\begin{equation*}
		\iii(P) = q^k \iii(P_0)
		\quad\text{and}\quad
		\abs{P} = q^k\abs{P_0} + \alpha \frac{q^k-1}{q-1}\,.
	\end{equation*}
\end{corollary}

\begin{corollary}\label{COR:lenghts-0-boundary}
	Let $\ell \in\NNN$ be such that $\kkk_\ell^0 \neq \emptyset$.
	Then there are $r \in \NNN$ and $0$-boundary line-patterns $P^1, \dots, P^r$ of length $\ell$
	such that
	$$ 
		\kkk_\ell^0 
		\ = \ 
		\bigsqcup_{j=1}^{r} \big(\iii (P^j) \times \{ 0\} \big)
		\ \sqcup  \ 
		\bigsqcup_{j=1}^{r} \big( \{ 0\}  \times \iii (P^j)\big). 
	$$ 
	Further, if $\ell \geq \recog$, $k$ is the order of every $P^j$, 
	and $P_i^j$ $(1 \leq j \leq r, \ 0 \leq i < k )$ are $0$-boundary line-patterns from Corollary~\ref{C:repeat-inducing-0-boundary} applied to $P^j$, then
	$$ 
		\kkk_\ell^0 = q^k \kkk_{\ell_0}^0,
		\qquad
		\kkk_{\ell_0}^0 
		\ =\  
		\bigsqcup_{j=1}^{r} \big(\iii (P_0^j) \times \{ 0\} \big)
		\ \sqcup \ 
		\bigsqcup_{j=1}^{r} \big( \{ 0\}  \times \iii (P_0^j)\big)
	$$ 
	and
	$$ 
		\ell = q^k \ell_0 + \alpha\frac{q^k-1}{q-1},
	$$
	where $\ell_0$ is the length of (every) $P_0^j$.
\end{corollary}

\section{Lines in finite symbolic recurrence plot $\recplot (n, \eps_0=1/2)$}\label{S:lengths-and-density-finite}

As in the previous section, let $\zeta$ be a uniform substitution satisfying assumptions \eqref{assumpt-start-with-0}--\eqref{assumpt-aperiodic}
and $x = \zeta^\infty(0)$ be the unique fixed point of it.
Now, our focus shifts to the comparison of \emph{infinite} $\recplot(\infty, \eps_0)$ and 
\emph{finite} $\recplot(n, \eps_0)$ symbolic recurrence plots of the sequence $x$ (recall that $\eps_0=1/2$).
We will show that there is a relatively small number of possible lengths of inner and $0$-boundary lines in the finite recurrence plot, see Propositions~\ref{PROP:maximal-number-of-lengths} and \ref{PROP:maximal-number-of-lengths-for-0-boundary};
moreover, every line starting in $\recplot(n, \eps_0)$ has its length (according to $\recplot(\infty, \eps_0)$)
bounded by $\recog n$.
Furthermore, we will show that the number of recurrences in $n$-boundary lines
grows at most linearly with $n$, see Proposition~\ref{PROP:recur-in-n-bound}.

\begin{proposition}[Lengths of inner lines]\label{PROP:maximal-number-of-lengths}
	Let $n \ge 2$.
	Then, in recurrence plot $\recplot(\infty, \eps_0)$, 
	the number of different lengths of inner lines starting in $[0,n)^2$
	is bounded from above by
	\begin{equation*}
		(\recog -1 )(1 + \log_q n)
	\end{equation*}
	and the length of every such line is smaller than $\recog n$.
\end{proposition}

\begin{proof}
	Let $\ell \geq \recog$ and assume that, in $\recplot(\infty, \eps_0)$,
	there is an inner line of length $\ell$ starting in
	$[0,n)^2$; that is, $\kkk_\ell \cap [1, n)^2\ne\emptyset$.
	Fix any $(i,j)\in\kkk_\ell \cap [1, n)^2$.
	Let $k\ge 1$ and $\ell_0<\recog$ be from Corollary~\ref{COR:lenghts};
	then 
	$\ell = q^k(\ell_0 + (\alpha+\beta)/(q-1))-  (\alpha+\beta)/(q-1) < \recog q^k$
	and $ \kkk_\ell = q^k \kkk_{\ell_0} - \beta(q^k-1)/(q-1)$.
	Hence there exists $(i_0, j_0) \in \kkk_{\ell_0}$ such that
	$ i = q^k i_0 - \beta(q^k-1)/(q-1)$
	and
	$ j = q^k j_0  - \beta(q^k-1)/(q-1)$.
	By the definition of $\kkk_{\ell_0}$ we have $i_0 \geq 1$, $j_0 \geq 1$ and
	$ i_0 \neq j_0$, so $i_0 \geq 2$ or $j_0 \geq 2$;
	we may assume the former.
	Then, using that $\beta \le q-1$,
	$$ 
		n > i =  q^k i_0 - \beta(q^k-1)/(q-1) > q^k,
	$$
	hence $k < \log_q n$.
	Thus, less than $\log_q n$ different lengths $\ell\ge\recog$
	of inner lines in $\recplot(\infty, \eps_0)$ starting in $[0,n)^2$ correspond to 
	fixed $\ell_0 < \recog$.
	From this fact we easily have that there are less than $(\recog-1)(1+\log_q n)$ different lengths of inner lines in $\recplot (\infty, \eps_0)$ starting in $[0,n)^2$.
	
	Finally, as was proved above, $\ell< \recog q^k$ and $q^k<n$, hence $\ell < \recog n$
	for every $\ell\ge\recog$ with $\kkk_\ell\cap [1,n)^2\ne\emptyset$.
	Since $\ell < \recog n$ is trivial for $\ell<\recog$, the proof is finished.
\end{proof}

\begin{proposition}[Lengths of $0$-boundary lines]\label{PROP:maximal-number-of-lengths-for-0-boundary}
	Let $n \ge 2$.
	Then, in recurrence plot $\recplot(\infty, \eps_0)$, the number of different lengths of $0$-boundary lines starting in $[0,n)^2$ is bounded from above by
	\begin{equation*}
		(\recog -1 )(1 + \log_q n)
	\end{equation*}
	and the length of every such line is less than $\recog n$.
\end{proposition}

\begin{proof}
	Let $\ell \geq \recog$ and assume that, in $\recplot(\infty, \eps_0)$, 
	there exists a $0$-boundary line of length $\ell$ starting in $[0,n)^2$; that is, $\kkk_\ell^0 \cap [0, n)^2\ne\emptyset$.
	Fix such a line $(i,j,\ell)$; we may assume that $j=0$. 
	By Corollary~\ref{COR:lenghts-0-boundary}, there exist $\ell_0 < \recog$ 
	and $k\ge 1$ such that
	$\ell = q^k(\ell_0 + \alpha/(q-1))- \alpha/(q-1)<\recog q^k$
	and $ \kkk_\ell^0 = q^k \kkk_{\ell_0}^0$.
	Hence $i_0=q^{-k}i$ is an integer and $(i_0,0) \in \kkk_{\ell_0}^0$.
	Since $i_0 \geq 1$, we have
	$$ 
		n > i =  q^k i_0  \geq q^k,
	$$
	thus $k < \log_q n$.
	Now the proof can be finished analogously as that of Proposition~\ref{PROP:maximal-number-of-lengths}.
\end{proof}

Consider now $n$-boundary lines in $\recplot(n,\eps_0)$. Every such line is a ``subset'' of a (possibly longer)
inner or $0$-boundary line in $\recplot(\infty,\eps_0)$ with the same starting point.
Though, by the previous results, the $\recplot(\infty,\eps_0)$-lengths of these lines are bounded from above by $\recog n$,
this does not imply that total number of recurrences in such lines is relatively small. It is the purpose of 
the following proposition to show that this is in fact true. 

We introduce some notation to distinguish lengths of lines in finite and infinite recurrence plots.
For an integer $n\ge 2$ and every $(i,j)\in[0,n)^2$ put
\begin{equation*}
	\ell_{ij}^n = \begin{cases}
		0 &\text{if there is no line in } \recplot(n,\eps_0) \text{ starting at } (i,j),
		\\
		\ell &\text{if there is a line in } \recplot(n,\eps_0) \text{ starting at } (i,j) \text{ of length }\ell,
	\end{cases}
\end{equation*}
and
\begin{equation*}
	\ell_{ij} = \begin{cases}
		0 &\text{if there is no line in } \recplot(\infty,\eps_0) \text{ starting at } (i,j),
		\\
		\ell &\text{if there is a line in } \recplot(\infty,\eps_0) \text{ starting at } (i,j) \text{ of length }\ell.
	\end{cases}
\end{equation*}
Clearly, for every $(i,j)\in[0,n)^2$, $\ell_{ij}^n\le \ell_{ij}$ and
\begin{equation}\label{EQ:len_ij}
	\begin{split}
		&\ell_{ij}^n = \ell_{ij}
		\quad\text{for every line }(i,j,\ell_{ij}^n) \text{ in } \recplot(n,\eps_0)
		\text{ which is not }n\text{-boundary},	
		\\
		&\ell_{ij}^n < \ell_{ij}	
		\quad\text{implies }(i,j,\ell_{ij}^n) \text{ is an }n\text{-boundary line in }\recplot(n,\eps_0).
	\end{split}
\end{equation}
For $n\ge 2$ and $\ell\in\NNN$,
let $b_\ell(n)$ denote the number of $n$-boundary lines $(i,j,\ell_{ij}^n)$ in $\recplot(n,\eps_0)$ such that $\ell_{ij}=\ell$. That is, for $\ell<n$, 
$$ 
b_\ell(n) = \card\Big(
(\kkk_\ell\sqcup\kkk_\ell^0) 
\cap 
\big([0, n)^2 \backslash [0, n-\ell)^2 \big)
\Big).
$$

\begin{proposition}[Recurrences in $n$-boundary lines]\label{PROP:recur-in-n-bound}
For every $\ell\ge 1$ and $n\ge 2$,
\begin{equation}\label{EQ:RR-and-densK:num_nboundary}
	\ell b_\ell(n) < 8\recog n.
\end{equation}
\end{proposition}

\begin{proof}
Since \eqref{EQ:RR-and-densK:num_nboundary} is trivially true for $\ell<\recog$, we may assume that
$\ell\ge\recog$.
By definition, every $n$-boundary line $(i,j,\ell_{ij}^n)$ contains a recurrence with
one coordinate equal to $n-1$; that is, there is $h\in[0,\ell_{ij}^n)$ such that
either $i+h=n-1$ or $j+h=n-1$
(the former case occurs when $i>j$, the latter one when $i<j$).
Further, every recurrence belongs to exactly one line.
Thus, by symmetry of recurrence plots,
\begin{equation}\label{EQ:RR-and-densK:b(n)2}
	b_\ell(n)=2 (\card B^\imath + \card B^0),
\end{equation}
where
\begin{eqnarray*}
	B^\imath &=& \{
	g\in[0,n-1)\colon
	(g,n-1) \text{ belongs to a line starting in } \kkk_\ell
	\},
	\\
	B^0 &=& \{
	g\in[0,n-1)\colon
	(g,n-1) \text{ belongs to a line starting in } \kkk_\ell^0
	\}.
\end{eqnarray*}

We prove that $\card B^\imath \le \recog(n-2)/\ell + 1$.
If $\kkk_\ell = \emptyset$ then $B^\imath=\emptyset$ and 
there is nothing to prove.
So assume that $\kkk_\ell \neq \emptyset$.
Then, by Corollary~\ref{COR:lenghts}, there are
$k \geq 1$ and $\ell_0 < \recog$ such that
$\kkk_\ell=q^k\kkk_{\ell_0}-\beta(q^k-1)/(q-1)$ and
$\ell = q^k(\ell_0 + c)-  c < \recog q^k$, where $c=(\alpha+\beta)/(q-1)$. 
Thus, for every $(i, j), (i', j')\in\kkk_\ell$,
\begin{equation}\label{EQ:RR-and-densK:b(n)-multiples}
	\abs{i-i'}, \abs{j - j'}, \abs{i-j} \text{ and }\abs{i'-j'}
	\quad\text{are multiples of }q^k.
\end{equation}
Take any $g<g'$ from $B^\imath$ and let $(i,j)$ and $(i', j') \in \kkk_\ell$
be starting points of $\ell$-lines in $\recplot(\infty,\eps_0)$ 
containing $(g,n-1)$ and $(g', n-1)$, respectively;
clearly, $i <j$ and $i' < j'$.
So there are $h$ and $h'$ from $[0, \ell)$
such that $j+h = n-1$ and $j' + h' = n-1$, respectively.
Further, by \eqref{EQ:RR-and-densK:b(n)-multiples},
there are $p, p' \geq 1$ such that $i = j-pq^k$ and $i' = j'-p'q^k$.
Since the lines starting at $(i,j)$ and $(i', j')$ are different hence do not have a common recurrence,
we have $p \neq p'$.
This gives
\begin{equation*}
	\abs{g-g'}
	=
	\abs{ (i+h) - (i'+h')} 
	= 
	\abs{p'-p} q^k 
	\geq q^k.
\end{equation*}
Since this is true for every different $g,g'\in B^\imath\subseteq[0,n-2]$, we immediately
have $\card B^\imath \le q^{-k}(n-2) + 1$.
Thus, using the facts that $q^k=(\ell+c)/(\ell_0+c)$, $\ell_0<\recog$ and $c\in[0,1]$,
\begin{equation*}
	\card B^\imath 
	\le 
	\frac{(\ell_0+c)(n-2)}{\ell+c} + 1
	\le
	\frac{\recog(n-2)}{\ell} + 1.
\end{equation*}

The proof of the fact that $\card B^0\le \recog(n-2)/\ell + 1$ is analogous, just instead of 
Corollary~\ref{COR:lenghts} one uses Corollary~\ref{COR:lenghts-0-boundary}.
Hence, by \eqref{EQ:RR-and-densK:b(n)2} and the fact that $\ell<\recog n$
by Propositions~\ref{PROP:maximal-number-of-lengths} and \ref{PROP:maximal-number-of-lengths-for-0-boundary},
$\ell b_\ell(n)	\le 	4\recog(n-2) + 4\ell	< 8\recog n$.
\end{proof}

\section{Algorithm for determining densities $\dens(\kkk_{\ell})$ for primitive substitutions}\label{S:densities}

In this section, we briefly discuss the algorithm for determining densities of the sets $\kkk_{\ell}$ for primitive substitutions and illustrate it on two examples.
A detailed algorithm can be found in Appendix~\ref{APP:algorithm}.

As a first step, we verify the assumptions for the substitution.
Then, we identify all allowed $2$-words ($\lang_2$).
Next, we proceed to identify all allowed $\ell$-words ($\lang_\ell$) 
for $\ell = 3, 4, \dots$ until
we reach a point where all allowed $\ell$-words are recognizable.
At this point,
we define $\recog = \max\{\ell, \alpha+\beta+1\}$
and then proceed to identify all allowed $(\recog+1)$-words.

Subsequently, for $\ell_0=1,2,\dots,\recog-1$, we determine all inner and $0$-boundary line-patterns.
For each line-pattern $\linepattern{a}{w}{b}$ with $\abs{w}<\recog$,
we calculate the measures $\mu([awb])$ and $\mu([\bar{a} w \bar{b}])$ using the algorithm
described in \cite[Section 5.4]{queffelec2010substitution}.

Finally, we compute the densities $\dens(\kkk_{\ell_0})$ for $1\le \ell_0<\recog$ using Theorem~\ref{THM:density}(\ref{Case2-in-THM:density}),
and $\dens(\kkk_{\ell})$ for $\ell\ge \recog$ using Theorem~\ref{THM:density}(\ref{Case3-in-THM:density}).

\begin{example}[Densities $\dens(\kkk_{\ell})$ for the Thue-Morse substitution]\label{example:dens-the-morse}
Let $\zeta$ be the Thue-Morse substitution $\zeta(0)=01$ and $\zeta(1)=10$.
As mentioned in Example~\ref{EXAM:def-recog}, $\recog=4$
since every allowed $4$-word is recognizable and $\alpha=\beta=0$.
All allowed words $w$ of length $3\le\ell\le \recog+1$ are (see also \cite[pp.~145--146]{queffelec2010substitution}):
\begin{itemize}
	\item 
	$001$, $010$, $011$, $100$, $101$, $110$; 
	\begin{itemize}
		\item[$\circ$] 
		$\mu([w])$ is always $1/6$;
	\end{itemize}
	\item 
	$0010$, $0011$, $0100$, $0101$, $0110$, $1001$, $1010$, $1011$, $1100$, $1101$; 
	\begin{itemize}
		\item[$\circ$] 
			$\mu([w])=1/6$ if $w\in\{0110,1001\}$;
		\item[$\circ$] 
			$\mu([w])=1/12$ otherwise;
	\end{itemize}
	\item
	$00101$, $00110$, $01001$, $01011$, $01100$, $01101$
	$10010$, $10011$, $10100$, $10110$, $11001$, $11010$; 
	\begin{itemize}
		\item[$\circ$] 
		$\mu([w])$ is always $1/12$.
	\end{itemize}
\end{itemize}
We have $8$ inner line-patterns $P=\linepattern{a}{w}{b}$ with $a=0$ and length $\abs{P}<\recog$:
\begin{itemize}
	\item 
	$\linepattern{0}{0}{1}$, $\linepattern{0}{1}{1}$;\quad
	$\linepattern{0}{01}{0}$, $\linepattern{0}{01}{1}$, $\linepattern{0}{10}{0}$,  $\linepattern{0}{10}{1}$;\quad
	$\linepattern{0}{010}{1}$, $\linepattern{0}{101}{1}$;
\end{itemize}
the densities of $\iii(P)$ equal $1/36$ for the first two patterns and $1/144$ for the remaining ones.
Note that this substitution is of type (\textendash+), see Lemma~\ref{LMM:on-induced-separators}.

Thus, by Theorem~\ref{THM:density}(\ref{Case2-in-THM:density}), densities of the sets $\kkk_\ell$ for $\ell < \recog$ are
\begin{equation*}
	\dens(\kkk_1)=\frac19,
	\quad
	\dens(\kkk_{2})=\frac1{18},
	\quad
	\dens(\kkk_{3})=\frac1{36}.
\end{equation*}
Finally, by Theorem~\ref{THM:density}(\ref{Case3-in-THM:density}),   densities of nonempty sets $\kkk_\ell$ are
\begin{equation*}
	\dens(\kkk_{2^{k+1}})=\frac1{18\cdot 2^{2k}},
	\quad
	\dens(\kkk_{3\cdot 2^{k}})=\frac1{36\cdot 2^{2k}}
\end{equation*}
for every $k\in\NNN$.

\end{example}

\begin{example}[Densities $\dens(\kkk_{\ell})$ for a substitution of length $5$]
Let $\zeta$ be the substitution given by $\zeta(0)=01110$ and $\zeta(1)=01010$; then $\recog=5$
since every allowed $4$-word is recognizable (see Example~\ref{EXAM:def-recog}) and $\alpha=\beta=2$.
Allowed words $w$ of length $3\le\ell\le \recog+1$ are:
\begin{itemize}
	\item 
	$001$, $010$, $011$, $100$, $101$, $110$, $111$; 
	\begin{itemize}
		\item[$\circ$] 
			$\mu([w])=1/5$ if $w\in\{001,010,100\}$;
		\item[$\circ$] 
			$\mu([w])=1/10$ otherwise;
	\end{itemize}
	\item 
	$0010$, $0011$, $0100$, $0101$, $0111$, $1001$, $1010$, $1100$, $1110$; 
	\begin{itemize}
		\item[$\circ$] $\mu([w])=1/5$ if $w=1001$;
		\item[$\circ$] $\mu([w])=1/10$ otherwise;
	\end{itemize}
	\item 
	$00101$, $00111$, $01001$, $01010$, $01110$, $10010$, $10011$, $10100$, $11001$, $11100$; 
	\begin{itemize}
		\item[$\circ$] $\mu([w])$ is always $1/10$;
	\end{itemize}
	\item 
	$001010$, $001110$, $010010$,  $010011$, $010100$, $011100$, $100101$, $100111$, $101001$, $110010$, $110011$, $111001$; 
	\begin{itemize}
		\item[$\circ$] 
		$\mu([w])=1/25$ if $w\in\{010010, 110011\}$;
		\item[$\circ$] 
		$\mu([w])=3/50$ if $w\in\{010011, 110010\}$;
		\item[$\circ$] 
		$\mu([w])=1/10$ otherwise.
	\end{itemize}
\end{itemize}
We have $10$ inner line-patterns $P=\linepattern{a}{w}{b}$ with $a=0$ and length $\abs{P}<\recog$:
\begin{itemize}
	\item 
	$\linepattern{0}{0}{1}$, $\linepattern{0}{1}{0}$, $\linepattern{0}{1}{1}$;
	\quad
	$\linepattern{0}{01}{1}$, $\linepattern{0}{10}{1}$, $\linepattern{0}{11}{1}$;
	\quad
	$\linepattern{0}{010}{1}$;
	\quad
	$\linepattern{0}{1001}{0}$, $\linepattern{0}{1001}{1}$;
\end{itemize}
the densities of the corresponding sets $\iii(P)$ are, respectively, 
\begin{itemize}
	\item 
	$1/25$, $1/100$, $1/50$;
	\quad
	$1/100$, $1/100$, $1/100$;
	\quad
	$1/100$;
	\quad
	$1/625$, $9/2500$.
\end{itemize}

Hence, the densities of nonempty sets $\kkk_\ell$ are:
\begin{equation*}
	\begin{split}
		&\dens(\kkk_{2\cdot 5^k-1})=\frac7{50\cdot 5^{2k}}\,,
		\quad
		\dens(\kkk_{3\cdot 5^k-1})=\frac3{50\cdot 5^{2k}}\,,
		\\	
		&\dens(\kkk_{4\cdot 5^k-1})=\frac1{50\cdot 5^{2k}}\,,
		\quad
		\dens(\kkk_{5\cdot 5^k-1})=\frac{13}{1250\cdot 5^{2k}}
	\end{split}
\end{equation*}
for every $k\in\NNN_0$. 
\end{example}

\section{Non-primitive binary substitutions}\label{S:non-primitive}
Now we will demonstrate the crucial role of the assumption of primitivity in our previous results.
For primitive substitutions, 
we have observed that the set of possible lengths of lines in the recurrence plot
is relatively small (has zero density in $\NNN$) and, for every possible length $\ell$,
the set $\kkk_\ell$ of starting points of $\ell$-lines has positive density in $\NNN_0^2$.
However, if we omit the assumption of primitivity, 
lines of every possible length $\ell$ appear in the recurrence plot, 
and the set $\kkk_\ell$ has always zero density.

In this section, consider a binary substitution $\zeta$ of constant length $q\ge 2$, 
which satisfies \eqref{assumpt-start-with-0} but not 
\eqref{assumpt-both-0-and-1-in-zeta-k}, and is such that the unique fixed point $x = \zeta^\infty(0)$ of
$\zeta$ starting with $0$ is not eventually $\sigma$-periodic.
That is, by \cite{seebold1988periodicity},
$\zeta$ is a binary substitution of constant length $q$ such that
\begin{itemize}
\item $\zeta(0)$ starts with $0$, and contains $1$ and at least one additional $0$;
\item $\zeta(1)=1^q$.
\end{itemize}
Clearly, $\zeta$ is not recognizable, since every word $1^\ell$ is not recognizable.

We define line-patterns and sets $\kkk_\ell$ as in Definition~\ref{DEF:line-pattern} and \eqref{EQ:Kell}, respectively.
Let $z\in[2,q)$ be the number of zeros in $\zeta(0)$. Then, for every $k$, $\zeta^k(0) $ contains exactly $z^k$ zeros.
Trivially, $\zeta^k(1)$ contains no zeros.
For $v \in \AAa^*$, put
$$ 
	N_v = \{ j \in \NNN : x_{[j, j+|v|)} = v\}.
$$ 

\begin{lemma}\label{LMM:nonperiodic}
	For every integer $\ell > 0$, put $w = 1^\ell$, $a=0$ and $b=1$.
	Then $\linepattern{a}{w}{b}$ is an inner line-pattern with infinite $\indexoflines{w}{a}{b}$.
	Consequently, the set $\kkk_\ell$ is infinite for every $\ell\in\NNN$.
\end{lemma} 

\begin{proof}
	Fix $k \in \NNN$ such that $q^k > \ell$ and put $l = q^k$.
	There exists infinitely many $j$ such that $x_{[j, j+l)} = 1^{l}$,
	$x_{j-1} = x_{j+l} = 0$.
	For every such $j$ we have $j\in N_{awb}$ and $j+l-\ell\in N_{\bar{a}w\bar{b}}$,
	thus $(j+1, j+l-\ell+1) \in \indexoflines{w}{a}{b}$.
\end{proof}

\begin{example}
	The substitution $\zeta$ given by $\zeta(0) = 010, \zeta(1) = 111$
	has two other kinds of inner line-patterns $\linepattern{a}{w}{b}$, 
	not mentioned in Lemma~\ref{LMM:nonperiodic}:
	\begin{itemize}
		\item $w = \zeta^k(101)$, $a=0$, $b=1$ for every $k \geq 0$,
		\item $w = \zeta^k(1)$, $a=b=0$ for every $k \geq 0$,
	\end{itemize}
	and their mirror images.
\end{example}

\begin{lemma}
The density of zeros in $x$ is $\dens(N_0) = 0$.
\end{lemma}

\begin{proof}	
	First realize that, for every $k,j\in\NNN_0$, the number of zeros in $x_{[jq^k,(j+1)q^k)}$
	is either $z^k$ or $0$.
	Now fix any $n\in\NNN$  and write $n = \sum_{k=0}^{p} a_{p-k} q^{k}$, 
	where $p\ge 0$, $a_0 \neq 0$ and $a_{p-k}$ is an integer from $[0,q)$ for every $k$. 
	Then the number of zeros in $x_{[0, n)}$ is bounded from above by $m_n = \sum_k a_{p-k}z^k$.
	A trivial estimate yields
	$$
		\frac{m_n}{n} 
		\leq 
		\frac{\sum_{k=0}^{p} (q-1) z^k}{q^p} 
		\leq
		\frac{(q-1)((q-1)^{p+1}-1)}{(q-2)q^{p}} \to 0 \quad\text{as } n \to \infty,
	$$
	hence $\dens(N_0) = 0$.
\end{proof}

\begin{corollary}
	Let $v$ be a subword of $x$ that contains at least one $0$. Then 
	$$
		\dens(N_v) = 0.
	$$
\end{corollary}

\begin{proof}
	By the assumption, there is $k \in [0, |v|)$ such that $v_k = 0$.
	Then $N_v \subseteq k + N_0$.
\end{proof}

\begin{proof}[Proof of Theorem~\ref{THM:nonprimit}]
Trivially, for a given length $\ell$, the number of different inner line-patterns of length $\ell$ is finite. For any inner line-pattern $\linepattern{a}{w}{b}$, exactly two out of four separators $a$, $b$, $\bar{a}$, $\bar{b}$ are zeros;
thus at least one of $awb$ and $\bar{a}w\bar{b}$ contains $0$. 
Since $\indexoflines{w}{a}{b}=(N_{awb}+1)\times(N_{\bar{a}w\bar{b}}+1)$,
the previous corollary and Lemma~\ref{LMM:nonperiodic} immediately imply Theorem~\ref{THM:nonprimit}.
\end{proof}

\section{Remaining uniform binary substitutions}\label{S:other-subs}

Now, we will deal with uniform binary substitutions which were not considered in the previous sections.

Assume first that $\zeta$ satisfies the assumption~\eqref{assumpt-start-with-0}. 
Let $x=\zeta^\infty(0)$ be the unique fixed point of $\zeta$ starting with $0$.
By \cite{seebold1988periodicity},
there are six mutually exclusive possibilities:
\begin{enumerate}
	\item \label{ITEM:case-non-injective}
	$\zeta(0)=\zeta(1)$;
	
	\item \label{ITEM:case-periodic}
	$q$ is odd, $\zeta(0)=(01)^s0$ and $\zeta(1)=(10)^s1$, where $q=2s+1$;
	
	\item \label{ITEM:case-0000}
	$\zeta(0)=0^q\ne\zeta(1)$;
	
	\item \label{ITEM:case-0111}
	$\zeta(0)=01^{q-1}$ and $\zeta(1)=1^q$;
	
	\item \label{ITEM:case-nonprimit}
	$\zeta(0)$ contains $1$ and at least one additional $0$, $\zeta(1)=1^q$;
	
	\item \label{ITEM:case-primit}
	$\zeta$ satisfies both \eqref{assumpt-both-0-and-1-in-zeta-k} and \eqref{assumpt-aperiodic}.
\end{enumerate}
Clearly, in the first three cases $x$ is $\sigma$-periodic and in the fourth case $x=01^\infty$ is eventually $\sigma$-periodic;
in all these four cases, the recurrence plot $\recplot(\infty,\eps_0)$ has no line of finite length and has infinitely many lines of
infinite length. Case \eqref{ITEM:case-nonprimit} was described in Theorem~\ref{THM:nonprimit}
and Case \eqref{ITEM:case-primit} was described in Theorem~\ref{THM:density}.

\medskip

Now assume that a uniform binary substitution $\zeta$ does not satisfy the assumption~\eqref{assumpt-start-with-0}; then $\zeta$ does not have a fixed point
starting with $0$.
Two cases can occur:
\begin{enumerate}
	\item 
	$\zeta(1)$ starts with $1$;
	
	\item 
	$\zeta(1)$ starts with $0$.
\end{enumerate}
In the first case, we can simply exchange the symbols $0$ and $1$ and proceed as above
(that is, we consider the recurrence plot of the unique fixed point $x=\zeta^\infty(1)$ of $\zeta$).
In the second case, $\zeta$ has no fixed point, but has two periodic points with period $2$, one starting with $0$ 
(denote it by $x$) and the other one starting with $1$. 
However, the substitution $\zeta^2$ satisfies \eqref{assumpt-start-with-0} and $x$ is its unique fixed point starting with $0$.
Thus we simply replace the substitution $\zeta$ by $\zeta^2$ and continue as above.


\appendix

\section{Recognizable words and $1$-cuttings}\label{APP:1cutting}

\subsection{General substitutions}

In this appendix, $A$ is any alphabet (that is, a finite set of cardinality at least $2$) containing $0$,
and $\zeta$ is any (not necessarily uniform) substitution such that 
\begin{itemize}
\item $\zeta$ is non-erasing, that is, $\zeta(a)\ne\emptyword$ for every $a\in A$;
\item $\zeta$ is injective, see \eqref{assumpt-one-to-one};
\item $\zeta(0)$ starts with letter $0$, see \eqref{assumpt-start-with-0};
\item $\abs{\zeta(0)}\ge 2$.
\end{itemize}
Then clearly $\zeta^n(0)$ is a proper prefix of $\zeta^{n+1}(0)$ for every $n\in\NNN_0$,
hence there is a unique fixed point $x\in\Sigma=A^{\NNN_0}$ 
of $\zeta$ starting with $0$. Notice that
\begin{equation}\label{E1-increasing}
	E_1 = 
	\{
		0=\abs{\zeta(x_{[0,0)})}
		< \abs{\zeta(x_{[0,1)})}
		< \dots < \abs{\zeta(x_{[0,j)})}
		< \abs{\zeta(x_{[0,j+1)})} < \dots
	\}
\end{equation}

Recall Definition~\ref{DEF:1-cutting} from Subsection~\ref{SUBS:recognizability}.
If $\CCc=[s,v_0,\dots,v_{k-1},t]$ is a $1$-cutting at the index $i$ of $w$ and $j\in\NNN_0$ is such that
\eqref{LAB-1cutting-v}--\eqref{LAB-1cutting-E} from Definition~\ref{DEF:1-cutting} is true, we say that
$\CCc$ is a $1$-cutting at the index $i$ of $w$ (\emph{with} $j$). 
For every nonempty word $w\in\lang_\zeta$ put
\begin{equation*}
	N_w = \{
		i\in\NNN_0\colon  x_{[i,i+\abs{w})} = w
	\} \ne\emptyset.
\end{equation*}

\begin{lemma}\label{L:1cutting-props}
	Let $w\in\lang_\zeta$ be nonempty and $i\in N_w$. 
	Let $\CCc=[s,v_0,\dots,v_{k-1},t]$ be a $1$-cutting at the index $i$ of $w$ (with $j$);
	put $l=\abs{\zeta(x_{[0,j)})} - \abs{s}$. Then
	\begin{enumerate}
	\item\label{LAB:1cutting-props-zetajx1}
		$\abs{\zeta(x_{[0,j-1)})} < l$;
	\item\label{LAB:1cutting-props-zetaj}
		$\abs{\zeta(x_{[0,j)})} \le l+\abs{w}$ and the equality holds if and only if $w=s$;
	\item\label{LAB:1cutting-props-cutting}
		let $\CCc'=[s',v_0',\dots,v_{k'-1}',t']$ be another $1$-cutting at the index $i$ of $w$ (with $j'$);
		if $j=j'$ and $\abs{s}=\abs{s'}$, then $\CCc'=\CCc$.
	\end{enumerate}
\end{lemma}
\begin{proof}
\eqref{LAB:1cutting-props-zetajx1} follows from the fact that $s$ is a proper suffix of $\zeta(x_{j-1})$; 
the rest is trivial.
\end{proof}

\begin{lemma}\label{L:1cutting-E1i}
	Let $w\in\lang_\zeta$ be nonempty and $i\in N_w$. 
	Let $\CCc=[s,v_0,\dots,v_{k-1},t]$ be a $1$-cutting at the index $i$ of $w$ (with $j$). Then
	\begin{enumerate}
	\item\label{LAB:1cutting-E1i-empty}
		$E_1\cap [i,i+\abs{w}) = \emptyset$ if and only if $w=s$ (and so $k=0$, $s=w$ and $t=\emptyword$);		
	\item\label{LAB:1cutting-E1i-nonempty}
		if $E_1\cap [i,i+\abs{w}) \ne \emptyset$ then 
		$\min \big(E_1\cap [i,i+\abs{w})\big) = \abs{\zeta(x_{[0,j)})}=i+\abs{s}$.
	\end{enumerate}
\end{lemma}
\begin{proof}
\eqref{LAB:1cutting-E1i-empty}
	If $w\ne s$ then $l-\abs{s}=|\zeta(x_{[0,j)})|\in E_1\cap [l,l+\abs{s})$ 
	by Lemma~\ref{L:1cutting-props}\eqref{LAB:1cutting-props-zetaj}, hence $E_1\cap [i,i+\abs{w}) \ne \emptyset$
	by \eqref{LAB-1cutting-E} from Definition~\ref{DEF:1-cutting}.
	On the other hand, if $w=s$ then $E_1\cap [l,l+\abs{w}) = \emptyset$ by \eqref{E1-increasing} and Lemma~\ref{L:1cutting-props}\eqref{LAB:1cutting-props-zetajx1}--\eqref{LAB:1cutting-props-zetaj},
	hence $E_1\cap [i,i+\abs{w}) = \emptyset$.
	
\eqref{LAB:1cutting-E1i-nonempty}
	Assume that $E_1\cap [i,i+\abs{w}) \ne \emptyset$.
	Then $w\ne s$ by \eqref{LAB:1cutting-E1i-empty}.
	So, by Lemma~\ref{L:1cutting-props}\eqref{LAB:1cutting-props-zetajx1} and \eqref{LAB:1cutting-props-zetaj},
	$\min \big(E_1\cap [l,l+\abs{w})\big) = \abs{\zeta(x_{[0,j)})} = l+\abs{s}$. Now it suffices to use 
	\eqref{LAB-1cutting-E} from Definition~\ref{DEF:1-cutting}.	
\end{proof}

\begin{lemma}\label{L:1cutting-existence-nonempty}
	Let $w\in\lang_\zeta$ be nonempty and $i\in N_w$. 
	If $E_1\cap [i,i+\abs{w}) \ne \emptyset$ then there exists a $1$-cutting at the index $i$ of $w$.
\end{lemma}
\begin{proof}
	Let $j\ge 0$ be the unique integer such that 
	$\abs{\zeta(x_{[0,j)})} = \min\big(E_1\cap[i,i+\abs{w}) \big)$
	and $k\ge 0$ be the maximal integer such that $\abs{\zeta(x_{[0,j+k)})}\le i+\abs{w}$.
	Put $g=\abs{\zeta(x_{[0,j)})}-i\ge 0$ and $f=i+\abs{w}-\abs{\zeta(x_{[0,j+k)})}\ge 0$. Then
	$s=x_{[i,i+g)}$ is a proper suffix of $\zeta(x_{j-1})$ 
	and $t=\zeta(x_{j+k})_{[0,f)}$ is a proper prefix of $\zeta(x_{j+k})$. 
	Clearly, the $(k+2)$-tuple $[s,v_0,\dots,v_{k-1},t]$, where $v_h=\zeta(x_{j+h})$ ($h\in[0,k)$),
	satisfies \eqref{LAB-1cutting-w}--\eqref{LAB-1cutting-st} from Definition~\ref{DEF:1-cutting}.
	The final property  \eqref{LAB-1cutting-E} follows trivially since, by the choice of $j$, 
	$l=\abs{\zeta(x_{[0,j)})} - g = i$.
\end{proof}

Recall that a letter $a\in A$ is allowed if $N_a \neq \emptyset$, that is, $a=x_j$ for some $j\in\NNN_0$.

\begin{lemma}\label{L:1cutting-existence-empty}
	Let $w\in\lang_\zeta$ be nonempty and $i\in N_w$. 
	Assume that $E_1\cap [i,i+\abs{w}) = \emptyset$. Then there exists a $1$-cutting at the index $i$ of $w$
	if and only if $w$ is a proper suffix of $\zeta(a)$ for some allowed letter $a\in A$.
\end{lemma}
\begin{proof}
	The implication from left to right follows from Lemma~\ref{L:1cutting-E1i}. To prove the converse implication,
	take $j\ge 1$ such that $a=x_{j-1}$. Then
	$\CCc=[s=w,t=\emptyword]$ is a $1$-cutting at the index $i$ of $w$ (with $j$).
\end{proof}

\begin{proposition}\label{P:1cutting-existence-and-uniqueness}
	Let $w\in\lang_\zeta$ be nonempty and $i\in N_w$. 
	Then there exists a $1$-cutting at the index $i$ of $w$ if and only if one of the following conditions is true:
	\begin{enumerate}
	\item $E_1\cap [i,i+\abs{w}) \ne \emptyset$;
	\item $E_1\cap [i,i+\abs{w}) = \emptyset$ and $w$ is a proper suffix of $\zeta(a)$ for some allowed letter $a\in A$.
	\end{enumerate}	
	If this is the case then a $1$-cutting at the index $i$ is unique.
\end{proposition}
\begin{proof}
	The equivalence follows from Lemmas~\ref{L:1cutting-existence-nonempty} and \ref{L:1cutting-existence-empty}.
	The uniqueness in the case when $E_1\cap [i,i+\abs{w}) = \emptyset$
	follows from Lemma~\ref{L:1cutting-E1i}\eqref{LAB:1cutting-E1i-empty};
	thus it suffices to show the uniqueness in the case when $E_1\cap [i,i+\abs{w}) \ne \emptyset$.
	To this end, assume that $E_1\cap [i,i+\abs{w}) \ne \emptyset$ and take any $1$-cuttings $\CCc=[s,v_0,\dots,v_{k-1},t]$
	and $\CCc'=[s',v_0',\dots,v_{k'-1}',t']$ at the index $i$ of $w$ (with $j$ and $j'$, respectively).
	By Lemma~\ref{L:1cutting-E1i}\eqref{LAB:1cutting-E1i-empty}, $s\ne w\ne s'$; thus, 
	by Lemma~\ref{L:1cutting-E1i}\eqref{LAB:1cutting-E1i-nonempty}, $\abs{\zeta(x_{[0,j)})}= \min\big( E_1\cap [i,i+\abs{w}) \big)  =\abs{\zeta(x_{[0,j')})}$.
	So $j=j'$. Further, $\abs{s}=\abs{s'}$ by Lemma~\ref{L:1cutting-E1i}\eqref{LAB:1cutting-E1i-nonempty}.
	Hence $\CCc=\CCc'$ by Lemma~\ref{L:1cutting-props}\eqref{LAB:1cutting-props-cutting}.
\end{proof}

As a corollary we obtain that every sufficiently long word $w$
has a unique $1$-cutting at the index $i$ for every $i\in N_w$.

\begin{corollary}\label{C:1cutting-existence-and-uniqueness}
	Let $w\in\lang_\zeta$ satisfies $\abs{w}\ge\max_a \abs{\zeta(a)}$. 
	Then, for every $i\in N_w$, there is a  unique $1$-cutting at the index $i$ of $w$.
\end{corollary}

\subsection{Uniform substitutions}

For uniform substitutions, the connection between recognizability of a word and uniqueness of a $1$-cutting
is described in Propositions~\ref{P:1cutting-uniform-necessary} and \ref{P:1cutting-uniform-sufficient}.
To abbreviate their formulations, we say that a nonempty word $w\in\lang_\zeta$:
\begin{itemize}
\item \emph{has no $1$-cutting} if there is no $1$-cutting at the index $i$ of $w$ for every $i\in N_w$;
\item \emph{has a weakly unique $1$-cutting} if there are $k\ge 0$ and a $(k+2)$-tuple $\CCc$ of words over $A$ such that
	\begin{itemize}
	\item $w$ has no $1$-cutting at the index $i_0$ of $w$ for some $i_0\in N_w$,
	\item $\CCc$ is a $1$-cutting at the index $i_1$ of $w$ for some $i_1\in N_w$, and
	\item for every $i\in N_w$, either $\CCc$ is a $1$-cutting or there is no $1$-cutting at the index $i$ of $w$;
	\end{itemize} 
\item \emph{has a strongly unique $1$-cutting} if there are $k\ge 0$ and a $(k+2)$-tuple $\CCc$ of words over $A$ such that $\CCc$ is a (unique) $1$-cutting at the index $i$ of $w$ for every $i\in N_w$;
\item \emph{has (at least) two $1$-cuttings} if there are $i,i'\in N_w$ and $1$-cuttings  $\CCc\ne \CCc'$ at the indexes $i,i'$ (respectively) of $w$.
\end{itemize}
By Proposition~\ref{P:1cutting-existence-and-uniqueness}, 
every nonempty $w\in \lang_\zeta$ satisfies exactly one of the four conditions. For nonempty $w\in\lang_\zeta$
put
\begin{equation*}
	P_w = \{p_i\colon i\in N_w\},
\end{equation*}
where $p_i$ is the unique integer from $[0,q)$ such that $i\equiv p_i \pmod q$.

We start with two lemmas.

\begin{lemma}\label{L:1cutting-uniform-E1i}
	Let $q\ge 2$ and $\zeta$ be a $q$-uniform substitution satisfying \eqref{assumpt-start-with-0} 
	and \eqref{assumpt-one-to-one}.
	Let $w\in\lang_\zeta$ be nonempty and $i\in N_w$. 
	If $p_i+\abs{w}\ge q$ then there is a $1$-cutting $\CCc_i$ at the index $i$ of $w$. 
	
	Moreover, for every $i,i'\in N_w$ with $p_i+\abs{w}\ge q$ and $p_{i'}+\abs{w}\ge q$,
	$\CCc_i=\CCc_{i'}$ if and only if $p_i=p_{i'}$.
\end{lemma}
\begin{proof}
	Put $\ell =\abs{w}$ and $p=p_i$. 
	If $p+\ell=q$ then clearly $\CCc=[w,\emptyword]$
	is a $1$-cutting at the index $i$. 
	
	If $p+\ell>q$ and $p>0$, fix $j\in\NNN$ such that $i=q(j-1)+p$. 
	Then $\min\big( E_1\cap [i,i+\ell) \big) = qj$. 
	Let $k\in\NNN_0$ be the maximal integer such that $q(j+k)\le i+\ell$.
	Put $s=x_{[i,i+q-p=qj)}$, $v_h=x_{[q(j+h),q(j+h+1))}=\zeta(x_{j+h})$ for $h\in[0,k)$, and $t=x_{[q(j+k),i+\ell)}$.
	Then $\CCc=[s,v_0,\dots,v_{k-1},t]$ is a $1$-cutting at the index $i$.
	
	Finally, assume that $p+\ell>q$ and $p=0$. Fix $j\in\NNN$ such that $i=qj$. 
	As above, let	$k\in\NNN_0$ be the maximal integer such that $q(j+k)\le i+\ell$,
	and put $s=\emptyword$, $v_h=x_{[q(j+h),q(j+h+1))}=\zeta(x_{j+h})$ for $h\in[0,k)$, and $t=x_{[q(j+k),i+\ell)}$.  
	Then $\CCc=[s,v_0,\dots,v_{k-1},t]$ is a $1$-cutting at the index $i$.
	
	\medskip 
	
	It remains to prove the second (``uniqueness'') statement.
	Fix any $i,i'\in N_w$ such that $p_i+\abs{w}\ge q$ and $p_{i'}+\abs{w}\ge q$.
	Let $\CCc_i=[s,v_0,\dots,v_{k-1},t]$ and  $\CCc_{i'}=[s',v_0',\dots,v_{k'-1}',t']$.
	Since $\zeta$ is $q$-uniform, we clearly have that $\CCc_i=\CCc_{i'}$ if and only if $\abs{s}=\abs{s'}$.
	Now it suffices to realize that $p_i+\abs{s}=q=p_{i'}+\abs{s'}$.
\end{proof}

\begin{lemma}\label{L:1cutting-uniform-short}
	Let $q\ge 2$ and $\zeta$ be a $q$-uniform substitution satisfying \eqref{assumpt-start-with-0} 
	and \eqref{assumpt-one-to-one}.
	Let $w\in\lang_\zeta$ be nonempty.
	Then, for every $i\in N_w$, $\CCc=[s=w,t=\emptyword]$ is a $1$-cutting at the index $i$ if and only if
	$\max P_w = q-\abs{w}$.
\end{lemma}
\begin{proof}
	Put $\ell=\abs{w}$.
	Assume first that $\max P_w < q-\ell$ and suppose that there is 
	$i\in N_w$ such that there is a $1$-cutting at the index $i$ of $w$.
	Clearly, $E_1\cap [i,i+\ell)=\emptyset$ and so, 
	by Lemma~\ref{L:1cutting-existence-empty}, $w$ is a proper suffix of $\zeta(a)$
	for some allowed letter $a$. But then $jq-\ell\in N_w$ for every $j\ge 1$ with $x_{j-1}=a$, thus $\max P_w = q-\ell$.
	This contradicts the assumption and shows that $w$ has no $1$-cutting at every index $i\in N_w$.		
	
	If there is $i\in N_w$ with $p_i>q-\ell$ then, by Lemma~\ref{L:1cutting-uniform-E1i}, there is
	a $1$-cutting $\CCc_i=[s,v_0,\dots,v_{k-1},t]$ at the index $i$. Clearly, $p_i+\abs{s}=q$, hence
	$\abs{s}<\ell$ and so $\CCc_i\ne \CCc$.
	
	Finally assume that $\max P_w = q-\abs{w}$. Fix $i_0\in N_w$ such that $p_{i_0}=q-\abs{w}$
	and let $j_0\ge 1$ be such that either $i_0=qj_0$ (if $p_{i_0}=0$) or $i_0=q(j_0-1)+p_{i_0}$ (if $p_{i_0}>0$).
	Clearly, $\CCc=[w,\emptyword]$ is a $1$-cutting at the index $i_0$ of $w$ (with $j_0$).
	Further, by the assumption, for any $i\in N_w$ we have that $E_1\cap [i,i+\ell)=\emptyset$.
	Since also $E_1\cap [qj_0-\ell,qj_0)=\emptyset$,
	$\CCc$ is a $1$-cutting at the index $i$ of $w$ (with $j_0$).
\end{proof}

\begin{proposition}\label{P:1cutting-uniform-necessary}
	Let $q\ge 2$ and $\zeta$ be a $q$-uniform substitution satisfying \eqref{assumpt-start-with-0} 
	and \eqref{assumpt-one-to-one}.
	Let $w\in\lang_\zeta$ be a (nonempty) \emph{recognizable} word and $p_w$ be such that $N_w\subseteq q\NNN_0+p_w$.
	\begin{enumerate}
	\item\label{LAB:P:1cutting-uniform-necessary-long}
		If $p_w+\abs{w}\ge q$ then $w$ has a strongly unique $1$-cutting.
	\item\label{LAB:P:1cutting-uniform-necessary-short}
		If $p_w+\abs{w} < q$ then $w$ has no $1$-cutting.
	\end{enumerate} 
\end{proposition}
\begin{proof}
	\eqref{LAB:P:1cutting-uniform-necessary-long} follows from Lemma~\ref{L:1cutting-uniform-E1i}
	and \eqref{LAB:P:1cutting-uniform-necessary-short} follows from Lemmas~\ref{L:1cutting-uniform-short} and \ref{L:1cutting-E1i}\eqref{LAB:1cutting-E1i-empty}.
\end{proof}

\begin{proposition}\label{P:1cutting-uniform-sufficient}
	Let $q\ge 2$ and $\zeta$ be a $q$-uniform substitution satisfying \eqref{assumpt-start-with-0} 
	and \eqref{assumpt-one-to-one}.
	Let $w\in\lang_\zeta$ be a nonempty \emph{non-recognizable} word.
	Put $p=\max P_w$ and $p'=\max(P_w\setminus \{p\})$.
	\begin{enumerate}
	\item\label{LAB:P:1cutting-uniform-sufficient-long}
		If $p'+\abs{w}\ge q$ then $w$ has at least two $1$-cuttings.
	\item\label{LAB:P:1cutting-uniform-sufficient-middle1}
		If $p'+\abs{w} < q < p+\abs{w}$ then $w$ has a weakly unique $1$-cutting.
	\item\label{LAB:P:1cutting-uniform-sufficient-middle2}
		If $p'+\abs{w} < q = p+\abs{w}$ then $w$ has a strongly unique $1$-cutting $\CCc=[w,\emptyword]$.
	\item\label{LAB:P:1cutting-uniform-sufficient-short}
		If $p+\abs{w} < q$ then $w$ has no $1$-cutting.
	\end{enumerate}
\end{proposition}
\begin{proof}
	Since \eqref{LAB:P:1cutting-uniform-sufficient-long} immediately follows from Lemma~\ref{L:1cutting-uniform-E1i}, and
	\eqref{LAB:P:1cutting-uniform-sufficient-middle2} and \eqref{LAB:P:1cutting-uniform-sufficient-short}
	immediately follow from Lemmas~\ref{L:1cutting-uniform-short}
	and \ref{L:1cutting-E1i}\eqref{LAB:1cutting-E1i-empty}, it suffices to prove 
	\eqref{LAB:P:1cutting-uniform-sufficient-middle1}.

	If $p_i=p$ then, by Lemma~\ref{L:1cutting-uniform-E1i},
	we have a $1$-cutting $\CCc_i$ at the index $i$; clearly,  $\CCc_i\ne [w,\emptyword]$.
	Moreover, $\CCc_i=\CCc_{i'}$ for every $i,i'$ with $p_i=p_{i'}=p$.
	On the other hand, if $p_i\ne p$ then $p_i+\abs{w}<q$ and, 
	by Lemma~\ref{L:1cutting-uniform-short},
	$\CCc=[w,\emptyword]$ is not a $1$-cutting at the index $i$.
	Thus, by Lemma~\ref{L:1cutting-E1i}\eqref{LAB:1cutting-E1i-empty},
	there is no $1$-cutting at the index $i$.
\end{proof}

\begin{corollary}\label{COR:recognizable}
	Let $q\ge 2$ and $\zeta$ be a $q$-uniform substitution satisfying \eqref{assumpt-start-with-0} 
	and \eqref{assumpt-one-to-one}.
	Let $w\in\lang_\zeta$ be such that $\abs{w}\ge q$.	Then:
	\begin{enumerate}
	\item $w$ is recognizable if and only if $w$ has a strongly unique $1$-cutting;
	\item $w$ is not recognizable if and only if $w$ has at least two $1$-cuttings.
	\end{enumerate}	
\end{corollary}

\subsection{Uniform substitutions --- examples}
In the following examples we further illustrate possible relations between recognizability
of a word $w$ and uniqueness of its $1$-cutting for $q$-uniform substitutions. 
Since for long enough $w$ the relation is straightforward by Corollary~\ref{COR:recognizable},
in all examples below we deal with words $w$ satisfying $\abs{w} < q$.

\begin{example}
	Let $q=4$, $\zeta(0)=0110$ and $\zeta(1)=0111$. For $w=00$ we have $P_w=\{3\}$.
	Thus $w$ is recognizable and, 
	by Proposition~\ref{P:1cutting-uniform-necessary}\eqref{LAB:P:1cutting-uniform-necessary-long}, 
	it has a strongly unique $1$-cutting $\CCc=[s=0, t=0]$.
\end{example}

\begin{example}
	Let $q=3$, $\zeta(0)=010$ and $\zeta(1)=000$. 
	
	For $w=1$ we have $P_w=\{1\}$. Thus $w$ is recognizable and has no $1$-cutting 
	by Proposition~\ref{P:1cutting-uniform-necessary}\eqref{LAB:P:1cutting-uniform-necessary-short}.
	
	For $w=0$ we have $P_w=\{0,1,2\}$.
	Thus $w$ is not recognizable, but it has a strongly unique $1$-cutting
	$\CCc=[s=w, t=\emptyword]$ by 
	Proposition~\ref{P:1cutting-uniform-sufficient}\eqref{LAB:P:1cutting-uniform-sufficient-middle2}.
	
	For $w=00$ we have $P_w=\{0,1,2\}$.
	Thus $w$ is not recognizable and, 
	by Proposition~\ref{P:1cutting-uniform-sufficient}\eqref{LAB:P:1cutting-uniform-sufficient-long}, 
	it has at least two $1$-cuttings; clearly, 
	\begin{itemize}
	\item $\CCc=[s=00, t=\emptyword]$ is a $1$-cutting at every index $i$ with $p_i\in \{0,1\}$;
	\item $\CCc=[s=0, t=0]$ is a $1$-cutting at every index $i$ with $p_i=2$.
	\end{itemize}\end{example}

\begin{example}
	Let $q=5$, $\zeta(0)=01110$ and $\zeta(1)=01010$. 
	
	For $w=10$ we have $P_w=\{1,3\}$.
	Thus $w$ is not recognizable, but it has a strongly unique $1$-cutting
	$\CCc=[s=w, t=\emptyword]$ by 
	Proposition~\ref{P:1cutting-uniform-sufficient}\eqref{LAB:P:1cutting-uniform-sufficient-middle2}.	
	
	For $w=01$ we have $P_w=\{0,2\}$; $w$ is not recognizable and has no $1$-cutting
	by Proposition~\ref{P:1cutting-uniform-sufficient}\eqref{LAB:P:1cutting-uniform-sufficient-short}.
\end{example}

\begin{example}
	Let $q=5$, $\zeta(0)=01101$ and $\zeta(1)=10000$. 
	For $w=011$ we have $P_w=\{0,3\}$.
	Thus $w$ is not recognizable, but it has a weakly unique $1$-cutting
	$\CCc=[s=w, t=\emptyword]$ by 
	Proposition~\ref{P:1cutting-uniform-sufficient}\eqref{LAB:P:1cutting-uniform-sufficient-middle1}.	
\end{example}

\section{Detailed algorithm for determining densities $\dens(\kkk_{\ell})$ for primitive substitutions} \label{APP:algorithm}
In the following we give a detailed description of the algorithm for determining densities $\dens(\kkk_{\ell})$
of the sets 
$\kkk_{\ell}$ of starting points of inner diagonal $\ell$-lines for primitive substitutions.  
Let $\zeta$ be a binary substitution of constant length $q$.

\begin{enumerate}[label=Step \arabic{enumi}]
	\item 
	Verify assumptions~\eqref{assumpt-start-with-0}--\eqref{assumpt-aperiodic} for
	the given substitution $\zeta$ (refer to the three bullet points in subsection~\ref{SUBS:substitution}).
	If $\zeta$ does not satisfy the conditions, finish.
	Otherwise, determine $\alpha$ and $\beta$ from Definition~\ref{DEF:alpha-and-beta}.
	
	\item \label{ITEM:step2}
	Identify $\lang_2$ (the set of allowed $2$-words) as follows.
	\begin{enumerate}[label=(\arabic*)]
		\item \label{ITEM:step2-1} Add all $2$-subwords of $\zeta(0)$ and $\zeta(1)$ to $\lang_2$ .
		\item \label{ITEM:step2-2} For every $ab\in\lang_2$, examine the $2$-word $cd=\zeta(a)_{q-1} \zeta(b)_0$;
		if $cd \notin \lang_2$, add $cd$ to $\lang_2$.
		\footnote{An easy argument shows that, by \eqref{assumpt-start-with-0}--\eqref{assumpt-aperiodic},
		it is not necessary to repeat \ref{ITEM:step2}\ref{ITEM:step2-2}.}
	\end{enumerate}

	\item   \label{ITEM:step3} 
	Determine $\recog$ and the sets $\lang_\ell$ $(2 < \ell \leq \recog +1 )$
	of all allowed $\ell$-words as follows.
	Initially, put $\recog = \infty$. Successively, for $\ell = 3, 4, \dots$:
	\begin{enumerate}[label=(\arabic*)]
		\item Identify the set $\lang_\ell$.\footnote{We use the fact that $\zeta$ is uniform; compare with \cite[Algorithm~3]{balchin2017computations}.}
		\begin{itemize}
			\item For $2<\ell\le q+1$, all $\ell$-words occur within words $\zeta(ab)_{[0, \ell+q-1)}$ ($ab\in\lang_2$).
			\item For $q+1<\ell\le 2q+1$, all $\ell$-words occur within words \newline $\zeta(abc)_{[0, \ell+q-1)}$ ($abc\in\lang_3$).
			\item Generally, if $(k-1)q+1 < \ell \leq kq+1$ for some $k \in \NNN$, all $\ell$-words
			occur within words $\zeta(v)_{[0,\ell+q-1)}$ ($v \in \lang_{k+1})$.
		\end{itemize}
		
		\item If $\recog = \infty$, determine whether all $\ell$-words $w$ are recognizable.
		\begin{itemize}
			\item Let $k \in \NNN_0$ be such that $(k-1)q+1 < |w| \leq kq+1$. Take the set $M_w$
			of all indexes $i$ such that $w$ starts at the index $i$
			in $\zeta(v)$ for some $v \in \lang_{k+1}$.
			If $M_w \pmod q$ is a singleton, $w$ is recognizable by Definition~\ref{DEF:recognizableWord}.
			\item If all $\ell$-words are recognizable, define $\recog = \max\{\ell, \alpha+\beta+1\}$.
		\end{itemize}
		
		\item If $\ell \geq \recog +1$, finish. Otherwise, continue with the next $\ell$.\footnote{For an upper bound of the number of iterations 
			in \ref{ITEM:step3} (that is, of $\recog-2$),
			see Remark~\ref{REM:recog-bounds}.}
	\end{enumerate}

	\item 	
	Determine all inner and $0$-boundary line-patterns of lengths $\ell_0=1, 2,$ $\dots,$ $\recog-1$.
	\begin{itemize}
		\item To determine inner line-patterns,
		realize that  $\linepattern{a}{w}{b}$ is an inner line pattern if and only if
		 both words $awb$ and $\bar{a} w \bar{b}$ are allowed.
		 
		\item To determine $0$-boundary line-patterns (not required for calculating densities $\dens(\kkk_{\ell})$, only for determining the types of lines present in the infinite recurrence plot), realize that for a $0$-boundary line-pattern $\linepattern{}{w}{b}$, both words $wb$ and $w \bar{b}$ must be allowed, with the additional requirement that $x$ starts with $w\bar{b}$.
	\end{itemize}
	
	\item 
	Calculate the measures $\mu([w])$ of cylinders $[w]$ for all allowed words $w$
	of length $3 \leq \ell \leq \recog +1$
	\cite[Section~5.4]{queffelec2010substitution}.
	\begin{itemize}
		\item Construct matrix $M_2 = (m_{ab, cd})_{ab, cd \in \lang_2}$
		such that $m_{ab, cd} $ is the number of occurrences of $ab$ in $\zeta(cd)_{[0, q+1)}$.
		\item Calculate the normalized Perron-Frobenius vector $\nu^{(2)} = \left(\nu_{ab}^{(2)}\right)_{ab \in \lang_2}$ of $M_2$; recall that the Perron-Frobenius value of $M_2$ is $q$.
		\item Successively, for $ \ell = 3,4, \dots, \recog +1$:\footnote{Actually, it suffices to accomplish this step only for $\ell = \recog+1$
					and then to use the fact that $[w] = [w0] \sqcup [w1] $ for every word $w$.}
			\begin{enumerate}[label=(\arabic*)]
			\item Determine the smallest integer $p$ such that $q^p > \ell-2$.
			\item Construct the matrix
					$M_{2, \ell,p} = (m_{w, ab})_{w \in \lang_\ell, ab \in \lang_2}$
					such that $m_{w, ab}$ is the number of
					occurrences of $w$ in $\zeta^p(ab)_{[0, q^p+\ell-1)}$.
			\item Let $\nu^{(\ell)} = \big(\nu^{(\ell)}_w\big)_{w \in \lang_\ell}$ be the normalized product of $ M_{2, \ell,p}$ and $\nu^{(2)}$; then, $\nu^{(\ell)}_w = \mu([w])$ for every $w \in \lang_\ell$.				
		\end{enumerate}
		
	\end{itemize}
	
	\item Calculate densities $\dens(\kkk_{\ell_0})$ for $1\le \ell_0<\recog$ using  Theorem~\ref{THM:density}(\ref{Case2-in-THM:density}), 
	and $\dens(\kkk_{\ell})$ for $\ell\ge \recog$ using Theorem~\ref{THM:density}(\ref{Case3-in-THM:density}).
\end{enumerate}


\section*{Acknowledgements}

The author is grateful for the numerous helpful suggestions provided by Vladim{\'\i}r {\v{S}}pitalsk{\'y}.
This work was supported by VEGA grant 1/0158/20.

\bibliography{subst-symbolic-rec-plot}

\providecommand{\bysame}{\leavevmode\hbox to3em{\hrulefill}\thinspace}
\providecommand{\MR}{\relax\ifhmode\unskip\space\fi MR }
\providecommand{\MRhref}[2]{%
  \href{http://www.ams.org/mathscinet-getitem?mr=#1}{#2}
}
\providecommand{\href}[2]{#2}
\begin{thebibliography}{10}

\bibitem{apparicio1999reconnaissabilite}
C.~Apparicio, \emph{Reconnaissabilité des substitutions de longueur
  constante}, 1999, Stage de Maîtrise de l’ENS Lyon.

\bibitem{balchin2017computations}
Scott Balchin and Dan Rust, \emph{Computations for symbolic substitutions},
  Journal of Integer Sequences \textbf{20} (2017).

\bibitem{bruin2022topological}
Henk Bruin, \emph{Topological and ergodic theory of symbolic dynamics}, vol.
  228, American Mathematical Society, 2022.

\bibitem{durand2017constant}
Fabien Durand and Julien Leroy, \emph{The constant of recognizability is
  computable for primitive morphisms}, Journal of Integer Sequences \textbf{20}
  (2017), no.~2, 3.

\bibitem{eckmann1987recurrence}
Jean-Pierre Eckmann, Sylvie~Oliffson Kamphorst, and David Ruelle,
  \emph{Recurrence plots of dynamical systems}, Europhysics Letters \textbf{4}
  (1987), no.~9, 973.

\bibitem{faure2010recurrence}
Philippe Faure and Annick Lesne, \emph{Recurrence plots for symbolic
  sequences}, International Journal of Bifurcation and Chaos \textbf{20}
  (2010), no.~06, 1731--1749.

\bibitem{fogg2002substitutions}
N~Pytheas Fogg, Val{\'e}r{\'e} Berth{\'e}, S{\'e}bastien Ferenczi, Christian
  Mauduit, and Anne Siegel, \emph{Substitutions in dynamics, arithmetics and
  combinatorics}, Springer, 2002.

\bibitem{klouda2016synchronizing}
Karel Klouda and Kate{\v{r}}ina Medkov{\'a}, \emph{Synchronizing delay for
  binary uniform morphisms}, Theoretical Computer Science \textbf{615} (2016),
  12--22.

\bibitem{marwan2023trends}
Norbert Marwan and K.~Hauke Kraemer, \emph{Trends in recurrence analysis of
  dynamical systems}, The European Physical Journal Special Topics (2023),
  1--23.

\bibitem{michel1976stricte}
Pierre Michel, \emph{Stricte ergodicit{\'e} d’ensembles minimaux de
  substitution}, Th{\'e}orie Ergodique: Actes des Journ{\'e}es Ergodiques,
  Rennes 1973/1974, Springer, 1976, pp.~189--201.

\bibitem{queffelec2010substitution}
Martine Queff{\'e}lec, \emph{Substitution dynamical systems-spectral analysis},
  vol. 1294, Springer, 2010.

\bibitem{seebold1988periodicity}
Patrice S{\'e}{\'e}bold, \emph{An effective solution to the {D}0{L} periodicity
  problem in the binary case}, Bulletin of the European Association for
  Theoretical Computer Science \textbf{36} (1988), 137--151.

\bibitem{vspitalsky2018recurrence}
Vladim{\'\i}r {\v{S}}pitalsk{\'y}, \emph{Recurrence quantification analysis of
  the period-doubling sequence}, International Journal of Bifurcation and Chaos
  \textbf{28} (2018), no.~14, 1850181.

\bibitem{tan2008periodicity}
Bo~Tan and Zhi-Ying Wen, \emph{Periodicity problem of substitutions over
  ternary alphabets}, RAIRO-Theoretical Informatics and Applications
  \textbf{42} (2008), no.~4, 747--762.

\bibitem{webber2015recurrence}
Charles~L. Webber and Norbert Marwan, \emph{Recurrence quantification analysis
  - theory and best practices}, Springer, 2015.

\bibitem{webber1994dynamical}
Charles~L. Webber~Jr and Joseph~P. Zbilut, \emph{Dynamical assessment of
  physiological systems and states using recurrence plot strategies}, Journal
  of applied physiology \textbf{76} (1994), no.~2, 965--973.

\bibitem{zbilut1992embeddings}
Joseph~P. Zbilut and Charles~L. Webber~Jr, \emph{Embeddings and delays as
  derived from quantification of recurrence plots}, Physics letters A
  \textbf{171} (1992), no.~3-4, 199--203.

\end{thebibliography}

\end{document}